%% file: main.tex
\documentclass[12pt]{article}

\usepackage{amsmath, amsthm, amssymb}
\usepackage{mathtools}
\usepackage{hyperref}
\usepackage{enumitem}
\usepackage{geometry}
\newtheorem{theorem}{Theorem}[section]
\newtheorem{lemma}[theorem]{Lemma}
\newtheorem{proposition}[theorem]{Proposition}
\newtheorem{corollary}[theorem]{Corollary}
\theoremstyle{definition}

\newcommand{\R}{\mathbb{R}}
\newcommand{\T}{\mathbb{T}}
\newcommand{\Z}{\mathbb{Z}}
\newcommand{\Ss}{\mathbb{S}^m}

\usepackage{comment}
\usepackage[most]{tcolorbox}
\newif\ifshowdetails

\usepackage{xcolor}
\newif\ifshowannotations
\showannotationstrue  

\usepackage{soul}
\sethlcolor{yellow}

\title{Thresholding Scheme for the \\ Half-Harmonic Map Heat Flow}
\author{Kilian Koch\thanks{Corresponding author.
  Lehrstuhl f\"ur Angewandte Analysis, RWTH Aachen University,
  Kreuzherrenstr.~2, 52062 Aachen, Germany.
  Email: \texttt{koch@math1.rwth-aachen.de}.
  Phone: +49\,241\,80-94586.}
\and
Christof Melcher\thanks{Lehrstuhl f\"ur Angewandte Analysis,
  RWTH Aachen University,
  Kreuzherrenstr.~2, 52062 Aachen, Germany.
  Email: \texttt{melcher@rwth-aachen.de}.
  Phone: +49\,241\,80-94585.}
\and
Endre S\"uli\thanks{Mathematical Institute, University of Oxford,
  Andrew Wiles Building, Radcliffe Observatory Quarter,
  Woodstock Road, Oxford OX2 6GG, United Kingdom.
  Email: \texttt{endre.suli@maths.ox.ac.uk}.}
}
\date{\today}

\begin{document}

\maketitle

\begin{abstract}

We propose a thresholding scheme for the half-harmonic map heat flow
from the flat torus $\T^d$ into the unit sphere $\Ss$ in $\R^{m+1}$, whose iterates are given by convolution with
the Poisson kernel, replacing the Gaussian kernel of the classical
Merriman--Bence--Osher algorithm, and pointwise normalization.
We prove that the piecewise constant interpolants subconverge to a
weak solution of the half-harmonic maps heat flow that attains the initial map strongly in the energy space and satisfies the energy-dissipation inequality.
For a spectrally truncated variant, convergence holds under a condition on the step size. A key ingredient is the derivation of the precise energy-dissipation identity directly from a refined within-step estimate, without resorting to De Giorgi interpolation typically used in the minimizing movement framework. This provides a direct way to recover the appropriate energy-dissipation property in the limit and, in turn, enables weak-strong uniqueness: every energy-dissipating weak solution coincides with the strong solution whenever the latter exists.
\end{abstract}

\section{Introduction and main results}
\label{sec:intro}

\input{introduction}



\section{Preliminaries}
\label{sec:prelim}

\input{preliminaries}

\section{Thresholding Scheme and Compactness}
\label{sec:compactness}

%

\input{compactness}

\section{Weak-strong uniqueness}
\label{sec:wsu}

%

\input{weak_strong}

\section*{Acknowledgements}
This work is funded by the Deutsche Forschungsgemeinschaft (DFG, German Research Foundation) -- project number \href{https://gepris.dfg.de/gepris/projekt/442047500}{442047500} -- through the Collaborative Research Center \href{https://sfb1481.rwth-aachen.de/}{``Sparsity and Singular Structures'' (SFB 1481)}.

\medskip\noindent
During the preparation of this manuscript, KK used Claude Opus 5
(Anthropic) as a discussion partner to explore and test various
mathematical ideas and strategies. All mathematical arguments and steps
presented in the manuscript were developed independently by the
authors. The authors take full responsibility for the content of the
manuscript.

\bibliographystyle{plain}
\bibliography{bib}

\end{document}

%% file: introduction.tex
The half-harmonic map heat flow is the $L^2$-gradient flow of the
half-Dirichlet energy
\[
  E(u)
  \;=\;
  \frac{c_d}{4}\int_{\T^d}\!\int_{\R^d}
    \frac{|u(x)-u(y)|^2}{|x-y|^{d+1}}\,dy\,dx
\]
for maps $u\colon \T^d\to\Ss$.
  The constant $c_d = \frac{\Gamma((d+1)/2)}{\pi^{(d+1)/2}}$
  is chosen so that
  \[
  E(u) = \tfrac{1}{2}\|(-\Delta)^{1/4}u\|_{L^2}^2
  \]
  where fractional Laplacians $(-\Delta)^{s/2}$ are defined via Fourier multipliers $|k|^s$ with $k \in \Z^d$. 
 Thus $E$ conincides with the usual homogeneous $\dot{H}^\frac{1}{2}$ norm. In abstract terms,
 the $L^2$ gradient flow equation takes the form
\[
\langle \partial_t u, \Phi \rangle_{L^2} + dE(u)\langle \Phi \rangle =0 
\]
for all sufficiently regular $\Phi=\Phi(x,t) \in \R^{d+1}$ with $\Phi \perp u$ almost everywhere.
Explicitly
\[
  \partial_t u + (-\Delta)^{1/2} u
  \;=\;
  \lambda \, u,
\]
with the Lagrange multiplier \[\displaystyle{\lambda(x)= \frac{c_d}{2}\int_{\R^d}
    \frac{|u(x)-u(y)|^2}{|x-y|^{d+1}}\,dy} \]
cf. e.g. \cite{MillotSire2015}.
In the framework of the fractional calculus introduced in \cite{Mazowiecka2018} $\lambda =|d_{\frac{1}{2}}u|^2_{\mathrm{od}}$
where $d_{\frac{1}{2}}$ denotes the fractional gradient and
$|\cdot|_{\mathrm{od}}$ the associated off-diagonal norm
(see Section~\ref{sec:prelim} for precise definitions), so that
\begin{equation}\label{eq:HHMHF}
  \partial_t u + (-\Delta)^{1/2} u
  = \left|d_{\frac{1}{2}}u \right|^2_{\mathrm{od}} \, u,
\end{equation}
revealing the formal analogy
with the classical harmonic map heat flow equation
\begin{equation} \label{eq:HMHF}
\partial_t u - \Delta u = |\nabla u|^2 u,
\end{equation}
corresponding to the $L^2$-gradient flow of the Dirichlet energy
$F(u)=\frac{1}{2}\int_{\T^d}|\nabla u|^2\,dx$.
Despite the formal similarity, \eqref{eq:HHMHF} has substantially different features, 
including a lower critical dimension $d=1$, rather than $d=2$
for the harmonic map heat flow, as well as new analytical challenges and open questions concerning singularity formation.

Critical points of $E$, known as half-harmonic maps,
were introduced by Da~Lio and Rivi\`ere~\cite{Francesca2011-us}, see also \cite{MillotSire2015, Moser2011}.
In the critical one-dimensional setting \(d=1\), it provides a genuinely 
nonlocal fractional counterpart to the conformally invariant harmonic 
problem in dimension \(d=2\). In this context, the half-harmonic map heat flow for maps
from \(\mathbb{S}^1\) was introduced and studied by
Wettstein~\cite{Wettstein2021,Wettstein2022,Wettstein2023}, who
established existence, uniqueness, and regularity results, including
local well-posedness in the critical space \(H^{1/2}\) and a bubbling
analysis at possible blow-up points. Independently, Struwe~\cite{Struwe2024} used the
Dirichlet-to-Neumann representation of the half-Laplacian and the
corresponding Plateau-flow interpretation, providing a detailed
description of singular times fully analogous to his classical result
for the harmonic map heat flow from surfaces~\cite{Struwe1985}.
He also proved uniqueness under additional regularity assumptions,
while Wright~\cite{Wright2024} subsequently extended uniqueness to
weaker solutions satisfying an almost-monotone energy condition. 
Concerning long-time behaviour, Sire, Wei, and Zheng~\cite{SireWeiZheng2021} 
constructed solutions exhibiting infinite-time blow-up for the flow from \(\mathbb{R}\) into
\(\mathbb{S}^1\), while the possibility of finite-time singularities
remains open. 

In higher spatial dimensions \(d>1\), existence and regularity results are
available for a large variety of related fractional geometric flows. Schikorra, Sire,
and Wang~\cite{SchikorraSireWang2017} established the existence of global
weak solutions for the geometric flow associated with their general
integro-differential energies, including the half-Dirichlet energy as a special case.
Their construction is based on a minimizing-movement scheme and applies, in
particular, to targets given by spheres and Riemannian homogeneous manifolds.
Existence and partial regularity of related space-time fractional flow has been studied by
Hyder, Segatti, Sire, and Wang~\cite{HyderSegattiSireWang2022}
starting from a Ginzburg-Landau approximation. Complementary small data global and
large data local well-posedness results for half-harmonic heat flows into spheres
in scaling critical function spaces were obtained in~\cite{MelcherSakellaris2019} and, more recently, 
in \cite{KochMelcher2025} using fixed point iterations based on the mild formulation of~\eqref{eq:HHMHF}.\\

In this work, we propose and analyze a new scheme for~\eqref{eq:HHMHF} that belongs to the family of thresholding algorithms originating from the work of Merriman, Bence, and Osher~\cite{MBO1992}. They introduced the diffuse-and-threshold procedure for the motion of hypersurfaces by mean curvature, which alternates diffusion by the heat kernel $G_h$ with pointwise thresholding. Building on the interpretation of this procedure as a minimizing-movements scheme by Esedoglu and Otto~\cite{EsedogluOtto2015}, Laux and Otto~\cite{LauxOt} established convergence to a weak formulation of mean curvature flow.
Laux and Yip~\cite{LauxYip2019} subsequently extended this framework to mean curvature flows in codimension two and to harmonic map heat flows~\eqref{eq:HMHF}, replacing the thresholding step by pointwise projection onto the sphere. The minimzing movement interpretation is based on the thresholding energy
\[
F_h(u) = \frac{1}{2h}\int_{\T^d} (1-u\cdot G_h*u)\,dx
\]
and the squared distance
\[
d^2_h(u,v)= \int_{\T^d} (u-v)\cdot G_h*(u-v)\,dx
\]
approximating the Dirichlet energy and $L^2$-distance, respectively, as $h\to 0$.
Replacing $G_h$ by the Poisson kernel $P_h$, the semigroup kernel of the half-Laplacian,
yields a scheme whose thresholding energy 
\begin{equation} \label{eq:aprox_energy}
 E_h(u) := \frac{1}{2h} \int_{\mathbb{T}^d}
    \bigl(1 - u \cdot P_h * u \bigr)\, dx,
\end{equation}
converges to the half-Dirichlet energy instead.
We therefore study the iteration
\begin{equation}\label{eq:scheme}
  u^{n+1} \;=\; \frac{P_h * u^n}{|P_h * u^n|}.
\end{equation}
Denoting the piecewise constant interpolant in time of the sequence $u^n$ by
$u^h(t) = u^n$ for $nh \le t < (n+1)h$ with discrete
time derivative $\partial_t^h u^h(t) = \frac{1}{h} (u^h(t+h) - u^h(t))$, so that
$\partial_t^h u^h = \frac{1}{h} (u^n - u^{n-1})$ for $(n-1)h \le t <nh$ and
$n = 1, \dots, \lfloor T/h \rfloor$, and a fixed terminal time $T > 0$,
we show the following convergence result:

\begin{theorem}
\label{thm:main}
Let $u^0 \in H^{1/2}(\mathbb{T}^d;\mathbb{S}^m)$.
Then there is a sequence $h_n \to 0$ such that the piecewise constant
interpolant $u^{h_n}$ of \eqref{eq:scheme} satisfies
\begin{align*}
  u^{h_n} &\to u && \text{in } L^2(\mathbb{T}^d \times (0,T)), \\
  (-\Delta)^{1/4} (P_{h_n/2} * u^{h_n})
    &\rightharpoonup (-\Delta)^{1/4} u
    && \text{in } L^2(\mathbb{T}^d \times (0,T)), \\
  \partial_t^{h_n} (P_{h_n/2} * u^{h_n})
    &\rightharpoonup \partial_t u
    && \text{in } L^2(\mathbb{T}^d \times (0,T)),
\end{align*}
for all $T>0$ as $n \to \infty$. The limit 
$u\in L^\infty(0,T;\,H^{1/2}(\mathbb{T}^d;\mathbb{S}^m))$ is a weak solution of \eqref{eq:HHMHF}
satisfying the energy-dissipation inequality
\begin{equation}\label{eq:energy-ineq-main}
  E(u(t)) + \int_0^t \|\partial_t u\|_{L^2}^2\,ds
  \;\leq\; E(u^0) \quad \text{for a.e.} \;\;t \in (0,T)
\end{equation}
and attaining the initial data $u^0$ strongly in $H^{1/2}(\mathbb{T}^d)$.
The claims hold true for a spectrally truncated variant in which $P_h$ is replaced by
its Fourier projection $P_h^N$ onto the modes $|k| \le N$ under the condition
\[
hN - d(d+1)\log(1/h) \to \infty.
\]
\end{theorem}

The proof builds on the conceptual framework established in \cite{LauxOt, LauxYip2019} with significant new challenges due to the nonlocality. A key new ingredient in the present setting is the derivation of an exact energy-dissipation identity for the thresholding scheme. This identity provides a direct proof of \eqref{eq:energy-ineq-main}, and in particular allows us to bypass the De Giorgi interpolation argument employed in \cite{LauxLelmi} to account for the contribution of the metric slope.

The spectrally truncated version lays the theoretical foundations of a numerical spectral-thresholding method
for~\eqref{eq:HHMHF} complementary to the tangent projection based method
proposed in \cite{Antil}. In the forthcoming work, a detailed numerical analysis and implementation will be presented \cite{KMNR}.

\medskip

Second, we establish a weak--strong uniqueness principle for~\eqref{eq:HHMHF}: 
\begin{theorem}\label{thm:wsu-main}
  Let $u^0 \in H^{1/2}(\mathbb{T}^d;\mathbb{S}^m)$, let $u \in L^\infty(0,T;\,H^{1/2}(\mathbb{T}^d;\mathbb{S}^m))$ with
$\partial_t u \in L^2(\mathbb{T}^d \times (0,T))$ be a weak
  solution of \eqref{eq:HHMHF} satisfying \eqref{eq:energy-ineq-main}, and let $v$
  be a strong solution, i.e., a weak solution satisfying additionally
  \[
    d_{\frac{1}{2}}v,\;d_{\frac{1}{2}}\partial_t v
      \in L^\infty(0,T;\,L^\infty_{\mathrm{od}}),
    \quad
    \partial_t v \in L^\infty(\mathbb{T}^d\times (0,T)),
  \]
  both with initial datum $u^0$.
  Then $u = v$ a.e.\ in $\mathbb{T}^d \times (0,T)$.
\end{theorem}


Our proof is based on the relative energy concept, as in \cite{FeireislNovotny, EmmrichLasarzik}. The structural parallels, namely, the sphere constraint $|u|=1$, an energy measuring the difference between two solutions, and a Gr\"onwall argument, carry over naturally to the setting of harmonic map heat flows. 
An analogous weak-strong uniqueness statement holds for the classical harmonic map heat flow equation \eqref{eq:HMHF}: if
  $u^0 \in H^1(\mathbb{T}^d;\mathbb{S}^m)$, $u$ is a weak
  solution and $v$ a strong solution with
  \[
    \nabla v,\;\nabla\partial_t v,\;\partial_t v
      \in L^\infty(\mathbb{T}^d\times (0,T)),
  \]
  then $u = v$ a.e.\ in $\mathbb{T}^d \times (0,T)$. The claim
  cannot be directly deduced from existing weak--strong uniqueness results for the Landau--Lifshitz--Gilbert equation~\cite{DumasSueur, DiFratta2020}, which rely essentially on the presence of the precessional term. 
  The argument rather follows the same structure as the proof of
  Theorem~\ref{thm:wsu-main}, with $\nabla$ in place of $d_{\frac{1}{2}}$;
  see Section~\ref{sec:wsu}.
Combining the convergence of the scheme with weak--strong uniqueness yields convergence of the piecewise constant interpolant of the iterates in~\eqref{eq:scheme} to the strong solution, as long as such a solution exists for~\eqref{eq:HHMHF}. Indeed, the weak solution obtained as the limit of the scheme satisfies the energy inequality, while the regularity of the strong solution provides precisely the additional assumptions required by the uniqueness theorem:

\begin{corollary}\label{cor:convergence-to-strong}
  If the assumptions of Theorem~\ref{thm:wsu-main} are satisfied,
  then the limit in Theorem~\ref{thm:main} is the strong solution.
\end{corollary}

The paper is organized as follows.
Section~\ref{sec:prelim} collects notation and preliminary material.
Section~\ref{sec:compactness} defines the thresholding scheme,
establishes compactness, and proves convergence to a weak solution.
Section~\ref{sec:wsu} contains the weak-strong uniqueness argument.

%% file: preliminaries.tex

\subsection{Fourier Series and Poisson kernel}

We work on the flat torus $\mathbb{T}^d = \mathbb{R}^d / (2\pi\mathbb{Z})^d$, where $u \in L^2(\mathbb{T}^d;\mathbb{R}^{m+1})$ admits a Fourier expansion
\[
  u(x) = \frac{1}{(2\pi)^{d/2}} \sum_{k \in \mathbb{Z}^d}
    \widehat{u}(k)\, e^{ik \cdot x}
  \quad \text{where} \quad
  \widehat{u}(k) = \frac{1}{(2\pi)^{d/2}}
    \int_{\mathbb{T}^d} u(x)\, e^{-ik \cdot x}\, dx.
\]

The \emph{Poisson kernel} $P_h$ on $\mathbb{T}^d$ is the convolution
kernel with Fourier coefficients
\[
  \widehat{P_h}(k) = \frac{1}{(2\pi)^{d/2}}\,e^{-h|k|},
  \qquad k \in \mathbb{Z}^d,
\]
i.e.\ $P_h$ is the kernel of the Poisson semigroup
$e^{-h(-\Delta)^{1/2}}$, so that
$\widehat{P_h * u}(k) = e^{-h|k|}\,\widehat{u}(k)$
and $\int_{\T^d} P_h\,dx = 1$.
In physical space, the Poisson kernel on $\mathbb{R}^d$ is
\[
  p_h(x) = c_d \frac{h}{(h^2 + |x|^2)^{(d+1)/2}},
  \qquad c_d = \frac{\Gamma(\frac{d+1}{2})}{\pi^{(d+1)/2}},
\]
and $P_h$ is its periodization,
$P_h(x) = \sum_{n \in \mathbb{Z}^d} p_h(x - 2\pi n)$.

The \emph{truncated Poisson kernel} $P_h^N$ is defined by
\[
  \widehat{P_h^N}(k) = \frac{1}{(2\pi)^{d/2}}\,
    e^{-h|k|}\,\mathbf{1}_{\{|k| \leq N\}},
  \qquad k \in \mathbb{Z}^d.
\]
The untruncated kernel is the limiting case $P_h = \lim_{N\to\infty} P_h^N$ (uniformly).
In particular, $P_h^N$ need not be non-negative. However,
under the growth condition on $N$ required in our main results,
positivity is guaranteed.

\begin{lemma}[Positivity of $P_h^N$]\label{lem:PhN_positivity}
For $h$ and $N = N(h)$ with $hN - d(d+1)\log(1/h) \to \infty$ as
$h \to 0$, we have $P_h^N(x) > 0$ for all $x \in \T^d$, provided $h$ is
small enough.
\end{lemma}

\begin{proof}
Since $P_h$ is the periodization of the positive free-space
kernel $p_h$, we have
$\min_{\T^d} P_h \geq m_d\,h$
for a constant $m_d > 0$ depending only on $d$.

For the truncation error, since
$|k|_\infty \leq |k| \leq \sqrt{d}\,|k|_\infty$,
we have
$\{k : |k| > N\} \subset \{k : |k|_\infty > N/\sqrt{d}\}$
and $e^{-h|k|} \leq e^{-(h/\sqrt{d})\,|k|_1}$, so
\[
  \sum_{|k| > N} e^{-h|k|}
  \;\leq\; \sum_{|k|_\infty > N/\sqrt{d}}
  e^{-(h/\sqrt{d})\,|k|_1}
  \;\leq\; S^d - S_N^d,
\]
where, writing $\tilde{h} := h/\sqrt{d}$,
\[
  S \;:=\; 1 + \frac{2e^{-\tilde{h}}}{1-e^{-\tilde{h}}},
  \qquad
  S_N \;:=\; 1 + \frac{2e^{-\tilde{h}}
  \bigl(1-e^{-\tilde{h}(N/\sqrt{d}-1)}\bigr)}{1-e^{-\tilde{h}}}
\]
are evaluated geometric series, using
$\lfloor N/\sqrt{d}\rfloor \geq N/\sqrt{d}-1$ for the integer cutoff of the
partial sum in $S_N$. Their difference is
\[
  S - S_N
  \;=\; \frac{2\,e^{-\tilde{h}N/\sqrt{d}}}{1-e^{-\tilde{h}}}
  \;=\; \frac{2\,e^{-hN/d}}{1-e^{-\tilde{h}}}.
\]
By $a^d - b^d \leq d\,a^{d-1}(a-b)$ for $0 \leq b \leq a$,
\[
  S^d - S_N^d
  \;\leq\; d\,S^{d-1}(S - S_N)
  \;\leq\; \frac{C_d}{h^d}\,e^{-hN/d},
\]
using $S \leq C/h$ and $1/(1-e^{-\tilde{h}}) \leq C/h$
for $h \leq 1$.
Since $hN - d(d+1)\log(1/h) \to \infty$,
the right-hand side is $o(h)$, hence
\[
  \|P_h - P_h^N\|_{L^\infty}
  \;\leq\; \frac{1}{(2\pi)^d}\sum_{|k|>N} e^{-h|k|}
  \;\leq\; \frac{C_d}{h^d}\,e^{-hN/d}
  \;\leq\; \tfrac{1}{2}\,m_d\,h
\]
for $h$ sufficiently small,
and $P_h^N \geq \tfrac{1}{2}\min P_h > 0$.
\end{proof}

\subsection{Fractional gradient and fractional Laplacian}

A suitable representation of the geometric nonlinearity arising from the half-Laplacian has been found in
\cite{MillotSire2015} which can be expressed in terms of the fractional calculus framework developed in \cite{Mazowiecka2018}, and which we adopt in our analysis. For a function $u : \mathbb{T}^d \to \mathbb{R}^{m+1}$, the
\emph{fractional gradient} $d_{\frac{1}{2}} u : \mathbb{T}^d \times \mathbb{R}^d
\to \mathbb{R}^{m+1}$ is defined by
\[
  d_{\frac{1}{2}} u(x,y) := \frac{\sqrt{c_d}}{\sqrt{2}}
    \frac{u(x) - u(y)}{|x - y|^{1/2}},
\]
where $c_d$ is as above. For
$f, g : \mathbb{T}^d \times \mathbb{R}^d \to \mathbb{R}^{m+1}$ we use the
\emph{off-diagonal product} and \emph{off-diagonal norm}
\[
  \langle f, g \rangle_{\mathrm{od}}(x)
  := \int_{\mathbb{R}^d} f(x,y) \cdot g(x,y)\, \frac{dy}{|x-y|^d},
  \qquad
  |f|_{\mathrm{od}}(x) := \sqrt{\langle f, f \rangle_{\mathrm{od}}(x)},
\]
both being functions of $x \in \mathbb{T}^d$ alone. For
$1 \leq p \leq \infty$ we set
\[
  \|f\|_{L^p_{\mathrm{od}}}
  := \bigl\|\,|f|_{\mathrm{od}}\,\bigr\|_{L^p(\mathbb{T}^d)},
  \qquad
  L^p_{\mathrm{od}}(\mathbb{T}^d;\mathbb{R}^{m+1})
  := \bigl\{ f : \|f\|_{L^p_{\mathrm{od}}} < \infty \bigr\} .
\]
For $p = 2$ this norm comes from the inner product
\[
  \langle f, g\rangle_{L^2_{\mathrm{od}}}
  := \int_{\mathbb{T}^d} \langle f, g\rangle_{\mathrm{od}}\,dx .
\]
In particular, $d_{\frac{1}{2}} u \in L^2_{\mathrm{od}}(\mathbb{T}^d)$ whenever
$u \in H^{1/2}(\mathbb{T}^d)$.
Note that $d_{\frac{1}{2}} u$ is antisymmetric: $d_{\frac{1}{2}} u(x,y) = -d_{\frac{1}{2}} u(y,x)$.
We have the integration by parts formula
\[
  \int_{\mathbb{T}^d} \langle d_{\frac{1}{2}} u, d_{\frac{1}{2}} v
    \rangle_{\mathrm{od}}\, dx
  = \int_{\mathbb{T}^d} (-\Delta)^{1/2} u \cdot v\, dx,
\]
where $(-\Delta)^{1/2}$ is the fractional Laplacian, defined as
the Fourier multiplier with symbol $|k|$.

A key structural property of half-harmonic maps is the equivalence
between the weak Euler--Lagrange equation and a nonlocal conservation
law.

\begin{theorem}[{\cite[Lemma 3.1]{Mazowiecka2018}}]
\label{thm:conservation_law}
Let $u \in H^{1/2}(\mathbb{T}^d;\mathbb{S}^m)$. Then the following
are equivalent:
\begin{enumerate}
  \item[\textup{(i)}] \emph{(Weak Euler--Lagrange equation.)}
    For every $\varphi \in H^{1/2} \cap L^\infty(\mathbb{T}^d;
    \mathbb{R}^{m+1})$,
    \[
      \int_{\mathbb{T}^d} \langle d_{\frac{1}{2}} u, d_{\frac{1}{2}} \varphi
        \rangle_{\mathrm{od}}\, dx
      = \int_{\mathbb{T}^d} |d_{\frac{1}{2}} u|^2_{\mathrm{od}}\,
        u \cdot \varphi\, dx.
    \]
  \item[\textup{(ii)}] \emph{(Nonlocal conservation law.)}
    For every $\zeta \in H^{1/2}(\mathbb{T}^d)$ and all $i,j \in \{1,\dots,m{+}1\}$,
    \[
      \int_{\mathbb{T}^d} \bigl(
        \langle d_{\frac{1}{2}} u_i, d_{\frac{1}{2}} \zeta
          \rangle_{\mathrm{od}}\, u_j
        - \langle d_{\frac{1}{2}} u_j, d_{\frac{1}{2}} \zeta
          \rangle_{\mathrm{od}}\, u_i
      \bigr)\, dx = 0.
    \]
\end{enumerate}
\end{theorem}

The same antisymmetric structure extends to the half-harmonic map
heat flow, and again in both directions.

\begin{theorem}
\label{thm:conservation_law_flow}
Let $u \in L^\infty(0,T;\,H^{1/2}(\mathbb{T}^d;\mathbb{S}^m))$ with
$\partial_t u \in L^2(\mathbb{T}^d\times(0,T))$. Then $u$ satisfies 
the weak formulation of \eqref{eq:HHMHF}
\begin{equation}\label{eq:weak-hhmhf-main}
  \int_0^T \!\int_{\mathbb{T}^d} \partial_t u \cdot \varphi\,dx\,dt
  \;+\;
  \int_0^T \!\int_{\mathbb{T}^d}
    \langle d_{\frac{1}{2}} u,\, d_{\frac{1}{2}} \varphi \rangle_{\mathrm{od}}\,dx\,dt
  \;=\;
  \int_0^T \!\int_{\mathbb{T}^d}
    |d_{\frac{1}{2}} u|^2_{\mathrm{od}}\,u \cdot \varphi\,dx\,dt
\end{equation}
for all $\varphi \in C^\infty(\mathbb{T}^d \times [0,T];\,\R^{m+1})$
if and only if 
\[
  \int_0^T \!\int_{\mathbb{T}^d} \bigl(
    (u_j\,\partial_t u_i - u_i\,\partial_t u_j)\,\zeta
    + \langle d_{\frac{1}{2}} u_i, d_{\frac{1}{2}} \zeta
      \rangle_{\mathrm{od}}\, u_j
    - \langle d_{\frac{1}{2}} u_j, d_{\frac{1}{2}} \zeta
      \rangle_{\mathrm{od}}\, u_i
  \bigr)\, dx\, dt = 0
\]
for every
$\zeta \in C_c^\infty(\mathbb{T}^d \times (0,T))$
and all $i,j \in \{1,\dots,m{+}1\}$.
\end{theorem}

\subsection{Approximative energy and approximative half-Laplacian}

The analysis of the thresholding scheme relies on the approximative energy
$E_h(u)$ in \eqref{eq:aprox_energy}
and its truncated variant
\[
  E_h^N(u) := \frac{1}{2h} \int_{\mathbb{T}^d}
    \bigl(1 - u(x) \cdot (P_h^N * u)(x)\bigr)\, dx,
\]
respectively.
Both can be viewed as discretizations of the energy $E$. Since $P_h^N$ has fewer Fourier modes than $P_h$,
we have $E_h(u) \leq E_h^N(u)$; see \eqref{eq:EmN-identity} below
for the exact difference. It readily follows (cf. \cite{LauxYip2019}) that the following result holds.

\begin{lemma}\label{lem:Eh_limit}
For any measurable $u : \mathbb{T}^d \to \mathbb{S}^m$,
\[
  E_h(u) \leq E(u)
  \qquad \text{and} \qquad
  \lim_{h \to 0} E_h(u) = E(u).
\]
Moreover, $E_h$ admits the representation
\[
  E_h(u) = \frac{1}{4h} \int_{\mathbb{T}^d} \int_{\mathbb{T}^d}
    P_h(x-y)\,|u(x) - u(y)|^2\, dy\, dx.
\]
\end{lemma}

The \emph{approximative half-Laplacian} is defined as
\[
  (-\Delta)^{1/2}_h u(x) := \frac{u(x) - (P_h * u)(x)}{h},
\]
the difference quotient of the Poisson semigroup, with Fourier
multiplier $\frac{1-e^{-h|k|}}{h}$.
Since $P_h$ is even, symmetrising the convolution gives
\[
  (-\Delta)^{1/2}_h u(x)
  = -\frac{1}{2h} \int_{\mathbb{T}^d} P_h(z)\,
    \bigl(u(x+z) - 2u(x) + u(x-z)\bigr)\, dz.
\]

We now turn to the truncated energy $E_h^N$. The following estimates
are straightforward modifications of the corresponding results for $E_h$.

\begin{lemma}\label{lem:EhN_limit}
For any measurable $u : \mathbb{T}^d \to \mathbb{S}^m$ and any $h > 0$, $N \in \mathbb{N}$,
\begin{align}\label{eq:EhN_upper_bound}
  E_h^N(u) \leq \Bigl(1 + \frac{1}{hN}\Bigr)\,E(u).
\end{align}
In particular, choosing $N = N(h)$ such that $hN\to \infty$
as $h \to 0$, we have
$E_h^N(u) \leq 2\,E(u)$ for all sufficiently
small $h$, and
\[
  \lim_{h \to 0} E_h^{N(h)}(u) = E(u).
\]
\end{lemma}

\begin{proof}
The only difference to the untruncated case is the additional tail
term. By Plancherel,
\[
  E_h^N(u)
  = \underbrace{\frac{1}{2h} \sum_{|k| \leq N}
      (1 - e^{-h|k|})\,|\widehat{u}(k)|^2}_{=:\,A_h}
    + \underbrace{\frac{1}{2h} \sum_{|k| > N}
      |\widehat{u}(k)|^2}_{=:\,B_h}.
\]
The term $A_h$ is estimated exactly as for $E_h$.
The tail $B_h$ satisfies
\[
  B_h \leq \frac{1}{2hN}
    \sum_{k} |k|\,|\widehat{u}(k)|^2
  = \frac{E(u)}{hN},
\]
which yields \eqref{eq:EhN_upper_bound} and vanishes whenever
$hN\to \infty$.
\end{proof}

\begin{lemma}\label{lem:Eh-def-truncated}
For every measurable $u : \mathbb{T}^d \to \mathbb{S}^m$,
\[
  E_h^N(u)
  = \frac{1}{4h} \int_{\mathbb{T}^d} \int_{\mathbb{T}^d}
    P_h^N(x - y)\,|u(x) - u(y)|^2\, dy\, dx.
\]
\end{lemma}

\begin{proof}
Since $|u|=1$, $|u(x)-u(y)|^2 = 2 - 2\,u(x)\cdot u(y)$. With
$\int_{\mathbb{T}^d} P_h^N(x-y)\,dy = 1$ and Fubini,
\[
  \int_{\mathbb{T}^d}\!\int_{\mathbb{T}^d}
    P_h^N(x-y)\,|u(x)-u(y)|^2\,dy\,dx
  = 2\int_{\mathbb{T}^d}\bigl(1 - u \cdot (P_h^N * u)\bigr)\,dx
  = 4h\,E_h^N(u). \qedhere
\]
\end{proof}

\begin{lemma}\label{lem:EmN-identity}
For any measurable $u : \mathbb{T}^d \to \mathbb{S}^m$, $h > 0$, and
$N \in \mathbb{N}$,
\begin{equation}\label{eq:EmN-identity}
  E_h^N(u) = E_h(u) + \frac{1}{2h}\sum_{|k| > N}
    e^{-h|k|}\,|\widehat{u}(k)|^2.
\end{equation}
\end{lemma}

\begin{proof}
By Plancherel,
\[
  E_h^N(u) - E_h(u)
  = \frac{1}{2h}\sum_{|k|>N}\bigl[|\widehat{u}(k)|^2
    - (1 - e^{-h|k|})|\widehat{u}(k)|^2\bigr]
  = \frac{1}{2h}\sum_{|k|>N} e^{-h|k|}\,|\widehat{u}(k)|^2. \qedhere
\]
\end{proof}

\begin{lemma}\label{lem:PhN-L2-approx}
For any $f \in L^2(\mathbb{T}^d)$ and $N = N(h)$ with
$hN \to \infty$,
\[
  \|P_h^N * f - f\|_{L^2} \to 0 \quad \text{as } h \to 0.
\]
\end{lemma}

\begin{proof}
By Plancherel,
\begin{align*}
  \|P_h^N * f - f\|_{L^2}^2
  &= \sum_{|k| \leq N} (e^{-h|k|} - 1)^2\,|\widehat{f}(k)|^2
    + \sum_{|k| > N} |\widehat{f}(k)|^2 \\
  &= \|P_h * f - f\|_{L^2}^2
    + \sum_{|k| > N} \bigl[1 - (e^{-h|k|} - 1)^2\bigr]
      \,|\widehat{f}(k)|^2.
\end{align*}
The first term vanishes as $h \to 0$. For the remainder,
$1 - (e^{-h|k|}-1)^2 = e^{-h|k|}(2-e^{-h|k|}) \leq 2e^{-hN}$
for $|k| > N$, so the sum is bounded by
$2e^{-hN}\|f\|_{L^2}^2 \to 0$.
\end{proof}

The \emph{truncated approximative half-Laplacian} is
\[
  (-\Delta)^{1/2}_{h,N} u(x)
  := \frac{u(x) - (P_h^N * u)(x)}{h}.
\]
Since $P_h^N$ is even, symmetrising the convolution gives
\[
  (-\Delta)^{1/2}_{h,N} u(x)
  = -\frac{1}{2h} \int_{\mathbb{T}^d} P_h^N(z)\,
    \bigl(u(x+z) - 2u(x) + u(x-z)\bigr)\,dz.
\]

%% file: compactness.tex

\subsection{Variational interpretation and energy bounds}

In this section we write $P_h^*$ to denote either the untruncated
kernel $P_h$ or the truncated kernel $P_h^N$, and correspondingly
$E_h^*$, $(-\Delta)^{1/2}_{h,*}$ for the associated energy and
approximative half-Laplacian. All results below hold for both choices
unless stated otherwise. The starting point is the variational formulation of the thresholding scheme, which is entirely analogous to the one used in \cite{LauxOt, LauxYip2019}. From the minimizing property, we obtain the following estimates. Since their proofs are straightforward consequences of the variational formulation, we omit the details.

\begin{proposition}
\label{lem:var-structure}
Each step of the thresholding scheme satisfies
\[
  \operatorname*{arg\,min}_{|v| \leq 1}\ \biggl(
    E_h^*(v) + \frac{1}{2h}\int_{\mathbb{T}^d}
      (v-u^{n-1})\cdot\bigl(P_h^* * (v-u^{n-1})\bigr)\,dx
  \biggr)
  \;=\; \frac{P_h^* * u^{n-1}}{|P_h^* * u^{n-1}|}
  \;=\; u^n ,
\]
which leads to the energy-dissipation estimate
\begin{align}\label{eq:energy_dissipation}
  E_h^*(u^h(\left\lfloor T/h \right\rfloor h)) + \tfrac12\int_0^{\left\lfloor T/h \right\rfloor h} \int_{\mathbb{T}^d}
    |P_{h/2}^* * \partial_t^h u^h|^2\, dx\, dt
  \;\leq\; E_h^*(u^0).
\end{align}
Let $u : \mathbb{T}^d \to \mathbb{S}^m$ be measurable. Then, $h \mapsto E_h^*(u)$ is
non-increasing and
\begin{align}\label{eq:HEstimate}
  E(P_{h/2}^* * u) \leq E_h^*(u).
\end{align}
\end{proposition}

The thresholding step $u^{n-1} \mapsto \frac{P_h^* * u^{n-1}}{|P_h^* * u^{n-1}|}$
requires $P_h^* * u^{n-1} \neq 0$ a.e. which is ensured for
\[
h < h_0 := \frac{(2 \pi)^d}{4 E(u^0)} \quad \text{and} \quad N \geq \frac{1}{h}.
\]
Indeed, assuming $P_h^N * u^0 = 0$ a.e., it follows that
\begin{align*}
    \frac{(2\pi)^d}{2h} = E_h^N(u^0) \leq 2 E(u^0) = \frac{(2\pi)^d}{2h_0}.
\end{align*}
A contradiction, thus by analyticity $P_h^N * u^0 \neq 0$ a.e. By
Lemma~\ref{lem:var-structure} we can repeat this argument inductively to
$u^n$. Thus the iteration is well defined. The untruncated case follows again by analyticity of $P_h * u$ and that
$P_h * u = 0$ is equivalent to $u = 0$, which cannot hold for $|u| = 1$.
\\\\
The estimate \eqref{eq:energy_dissipation} is not sharp and lacks, in the language of minimizing movements,
the contribution from the metric slope. However, the energy drop per
step in fact satisfies an exact identity, which we establish next.

\begin{theorem}[Energy-dissipation identity]\label{thm:energy-identity}
For every $k \geq 1$, writing $t_k := kh$ and $u^h(t_k) = u^k$, the iterates of
the thresholding scheme satisfy
\[
  E_h^*(u^h(t_k))
  + \frac12\int_0^{t_k}\!\!\int_{\mathbb{T}^d}
    |P_{h/2}^* * \partial_t^h u^h|^2\,dx\,dt
  + \frac12\int_0^{t_k}\!\!\int_{\mathbb{T}^d}
    |P_h^* * u^h|\,|\partial_t^h u^h|^2\,dx\,dt
  = E_h^*(u^0).
\]
Both dissipation terms are nonnegative; dropping the second recovers
\eqref{eq:energy_dissipation}.
\end{theorem}

\begin{proof}
Inserting $E_h^*(u) = \frac{1}{2h}\int_{\mathbb{T}^d}
\bigl(1 - u\cdot(P_h^* * u)\bigr)\,dx$ twice and expanding,
\begin{align}\label{eq:energy-split}
  E_h^*(u^{n-1}) - E_h^*(u^n)
  &= \frac{1}{2h}\int_{\mathbb{T}^d}\bigl(u^n\cdot(P_h^* * u^n)
     - u^{n-1}\cdot(P_h^* * u^{n-1})\bigr)\,dx \notag\\
  &= \frac{1}{2h}\int_{\mathbb{T}^d}(u^n-u^{n-1})\cdot
     \bigl(P_h^* * (u^n-u^{n-1})\bigr)\,dx \notag\\
  &\quad + \frac{1}{h}\int_{\mathbb{T}^d}(u^n-u^{n-1})\cdot
     (P_h^* * u^{n-1})\,dx.
\end{align}
With $u^n-u^{n-1} = h\,\partial_t^h u^h$ and
$P_h^* = P_{h/2}^* * P_{h/2}^*$, the first term of \eqref{eq:energy-split} is
\[
  \frac{1}{2h}\|P_{h/2}^* * (u^n-u^{n-1})\|_{L^2}^2
  = \frac{h}{2}\,\|\partial_t^h(P_{h/2}^* * u^h)\|_{L^2}^2.
\]
The scheme gives $P_h^* * u^{n-1} = |P_h^* * u^{n-1}|\,u^n$,
and $|u^n| = |u^{n-1}| = 1$ gives
$(u^n - u^{n-1})\cdot u^n = \tfrac12|u^n-u^{n-1}|^2$, so the second term of
\eqref{eq:energy-split} is
\[
  \frac{1}{2h}\int_{\mathbb{T}^d}|P_h^* * u^{n-1}|\,|u^n-u^{n-1}|^2\,dx
  = \frac{h}{2}\int_{\mathbb{T}^d}|P_h^* * u^{n-1}|\,|\partial_t^h u^h|^2\,dx.
\]
Summation over $n = 1,\dots,k$ gives the claim.
\end{proof}

\subsubsection*{Euler-Lagrange and approximate evolution equation}
With $\displaystyle{d_h^2(u,v)= \int_{\T^d} (u-v) \cdot P_h^*(u-v) \, dx}$,
the iterate $u^n$ minimizes
\[
u \mapsto E_h^*(u) + \frac{1}{2h}d_h^2(u, u^{n-1})
\] among $|u| \leq 1$ and satisfies
$|u^n| = 1$, so it is a critical point with respect to all tangential
variations $\delta u = (\mathrm{Id}-u^n\otimes u^n)\xi$ 
where $\xi \in C^\infty(\mathbb{T}^d;\R^{m+1})$. The piecewise constant interpolant $u^h$
therefore satisfies
\[
  \int_{\mathbb{T}^d}
    \bigl(P_h^* * \partial_t^h u^h(\cdot - h) + (-\Delta)^{1/2}_{h,*} u^h\bigr)
    \cdot (\mathrm{Id} - u^h\otimes u^h)\,\xi\,dx = 0,
\]
i.e., the approximate evolution equation
\begin{equation}\label{eq:approx_evol}
     (\mathrm{Id} - u^h \otimes u^h)\,\bigl(
    P_h^* * \partial_t^h u^h(\cdot - h)
    + (-\Delta)^{1/2}_{h,*} u^h
  \bigr) = 0
  \quad \text{in } \mathbb{T}^d \times (h,T).
\end{equation}
To avoid the Lagrange multiplier arising from the projection we antisymmetrize:

\begin{lemma}\label{lem:approx-EL}
The interpolant $u^h$ satisfies 
\begin{align*}
  \int_h^T \!\int_{\mathbb{T}^d} \Big(
    &\;u_j^h\,(P_h^* * \partial_t^h u_i^h(\cdot - h))
    - u_i^h\,(P_h^* * \partial_t^h u_j^h(\cdot - h)) \\
    &+ u_j^h\,(-\Delta)^{1/2}_{h,*} u_i^h
    - u_i^h\,(-\Delta)^{1/2}_{h,*} u_j^h
  \Big)\,\zeta\, dx\, dt = 0
\end{align*}
for all $\zeta \in C^\infty(\mathbb{T}^d \times [0,T])$ and all
$i,j \in \{1,\dots,m{+}1\}$. Moreover, for $\zeta \in L^\infty(\mathbb{T}^d)$,
\begin{align*}
  &\int_{\mathbb{T}^d} \bigl(
    u_j^h\,(-\Delta)^{1/2}_{h,*} u_i^h
    - u_i^h\,(-\Delta)^{1/2}_{h,*} u_j^h
  \bigr)\,\zeta\, dx \\
  &\qquad = \frac{1}{2h} \int_{\mathbb{T}^d} P_h^*(z)
    \int_{\mathbb{T}^d} \bigl[
      ((u_i^h)^z - u_i^h)\,u_j^h - ((u_j^h)^z - u_j^h)\,u_i^h
    \bigr]\,(\zeta^z - \zeta)\, dx\, dz,
\end{align*}
where $u^z(x) := u(x+z)$ for $z \in \mathbb{T}^d$.
\end{lemma}

\begin{proof}

Writing $F^h := P_h^* * \partial_t^h u^h(\cdot-h) + (-\Delta)^{1/2}_{h,*} u^h$,
\eqref{eq:approx_evol} says $(\mathrm{Id}-u^h\otimes u^h)F^h = 0$,
i.e.\ $F^h = (u^h\cdot F^h)\,u^h$, so $F^h$ is \emph{parallel} to $u^h$
with scalar factor $\lambda^h := u^h\cdot F^h$. Consequently, for any two
indices $i,j$,
\[
  u_j^h F_i^h - u_i^h F_j^h
  = \lambda^h\bigl(u_j^h u_i^h - u_i^h u_j^h\bigr) = 0
\]
pointwise. 
Writing out $F^h$ and testing against
$\zeta \in C^\infty(\mathbb{T}^d\times[0,T])$, then integrating over
$\mathbb{T}^d\times(h,T)$, gives the stated identity. 
\end{proof}

\subsection{Passing to the limit $h \to 0$}

We now pass to the limit $h \to 0$ in the approximative conservation
law (Lemma~\ref{lem:approx-EL}). The goal is to show
that the antisymmetric term converges to its continuous counterpart
involving $d_{\frac{1}{2}}$. This is where the untruncated and truncated
schemes require separate arguments.

By virtue of Lemma~\ref{lem:approx-EL}, the key quantity to analyze is
\[
  \frac{1}{2h} \int_{\mathbb{T}^d} P_h^*(z)
    \int_{\mathbb{T}^d} \bigl[
      ((u_i^h)^z - u_i^h)\,u_j^h - ((u_j^h)^z - u_j^h)\,u_i^h
    \bigr]\,(\zeta^z - \zeta)\, dx\, dz.
\]
The formal limit $P_h(z)/h \to c_d/|z|^{d+1}$ suggests convergence
to the fractional gradient expression
\[
  \int_{\mathbb{T}^d} \bigl(
    \langle d_{\frac{1}{2}}(P_h^* * u_i^h), d_{\frac{1}{2}} \zeta
      \rangle_{\mathrm{od}}\, u_j^h
    - \langle d_{\frac{1}{2}}(P_h^* * u_j^h), d_{\frac{1}{2}} \zeta
      \rangle_{\mathrm{od}}\, u_i^h
  \bigr)\, dx.
\]
We make this rigorous via Fourier analysis.

\begin{lemma}[Fourier representation]\label{lem:fourier-rep}
For measurable $u : \mathbb{T}^d \to \mathbb{S}^m$ and
$\zeta \in C^\infty(\mathbb{T}^d)$,
\begin{align*}
  \frac{1}{2h} \int_{\mathbb{T}^d} &P_h^*(z)
    \int_{\mathbb{T}^d} \bigl[
      (u_i^z - u_i)\,u_j - (u_j^z - u_j)\,u_i
    \bigr]\,(\zeta^z - \zeta)\, dx\, dz \\
  &= \frac{1}{2(2\pi)^{d/2}}
    \sum_{k,n \in \mathbb{Z}^d}
    A_{ij}(k,n)\,\widehat{\zeta}(n)\, K_h^*(k,n),
\end{align*}
where $A_{ij}(k,n) = \widehat{u}_i(k)\,\widehat{u}_j(-k-n)
- \widehat{u}_j(k)\,\widehat{u}_i(-k-n)$ and the kernel is
\[
  K_h(k,n) = \frac{e^{-h|k+n|} - e^{-h|k|} - e^{-h|n|} + 1}{h}
\]
for the untruncated scheme, respectively
\[
  K_h^N(k,n) = \frac{
    e^{-h|k+n|}\mathbf{1}_{\{|k+n| \leq N\}}
    - e^{-h|k|}\mathbf{1}_{\{|k| \leq N\}}
    - e^{-h|n|}\mathbf{1}_{\{|n| \leq N\}}
    + 1}{h}
\]
for the truncated scheme.
\end{lemma}

\begin{proof}
The Fourier series of $u$ converges in $L^2(\mathbb{T}^d)$, while
that of $\zeta$ converges in $L^\infty(\mathbb{T}^d)$; by limit
interchange we obtain
\[
  \int_{\mathbb{T}^d} (u_i^z - u_i)\,u_j\,(\zeta^z - \zeta)\, dx
  = (2\pi)^{-d/2}
    \sum_{\substack{k,\ell,n \\ k+\ell+n=0}}
    (e^{ik \cdot z} - 1)(e^{in \cdot z} - 1)\,
    \widehat{u}_i(k)\,\widehat{u}_j(\ell)\,\widehat{\zeta}(n).
\]

The second term $(i \leftrightarrow j)$ has the same form with
exchanged indices. Setting $\ell = -k-n$ gives
\begin{gather*}
  \int_{\mathbb{T}^d} \bigl[
    (u_i^z - u_i)\,u_j - (u_j^z - u_j)\,u_i
  \bigr]\,(\zeta^z - \zeta)\, dx \\
  = (2\pi)^{-d/2} \sum_{k,n}
    (e^{ik \cdot z} - 1)(e^{in \cdot z} - 1)\,
    A_{ij}(k,n)\,\widehat{\zeta}(n).
\end{gather*}
Integration against $\frac{1}{2h}P_h^*(z)$ in $z$ and interchanging
summation and integration by virtue of the Fubini--Tonelli theorem yields
\[
  \frac{1}{2h} \int_{\mathbb{T}^d} P_h^*(z)\,
    (e^{ik \cdot z} - 1)(e^{in \cdot z} - 1)\, dz
  = \frac{1}{2h} \bigl(
    e^{-h|k+n|} - e^{-h|k|} - e^{-h|n|} + 1
  \bigr)
\]
for the untruncated kernel, and correspondingly the same expression with indicator
functions $\mathbf{1}_{\{|\cdot|\leq N\}}$ for the truncated one.
This gives $K_h$ and correspondingly $K_h^N$.

\end{proof}

The following identity relates $K_h$ to the kernel
$|k+n| - |k| - |n|$ in the limit $h \to 0$.

\begin{lemma}\label{lem:integral_identity}
For all $k,n \in \mathbb{R}^d$,
\[
  \int_{\mathbb{R}^d}
    \frac{e^{i(k+n) \cdot z} - e^{ik \cdot z}
      - e^{in \cdot z} + 1}{|z|^{d+1}}\, dz
  = -c_d^{-1}\bigl(|k+n| - |k| - |n|\bigr).
\]
\end{lemma}

\begin{proof}
The integral converges absolutely.

By symmetry
$\int_{\R^d} f(z)\,dz = \frac{1}{2}\int_{\R^d}(f(z)+f(-z))\,dz$ and
$e^{-i\theta} = \overline{e^{i\theta}}$, the left-hand side equals
\[
  \int_{\mathbb{R}^d}
    \frac{\cos((k+n) \cdot z) - \cos(k \cdot z)
      - \cos(n \cdot z) + 1}{|z|^{d+1}}\, dz.
\]
Define
\[
  J(\xi) := \int_{\mathbb{R}^d}
    \frac{1 - \cos(\xi \cdot z)}{|z|^{d+1}}\, dz,
\]
so that the integral equals $-J(k+n) + J(k) + J(n)$.
For any orthogonal matrix $R$, the substitution $z \mapsto R^\top z$
gives $J(R\xi) = J(\xi)$, so $J$ depends only on $|\xi|$.
For $\lambda > 0$, the substitution $u = \lambda z$ gives
\[
  J(\lambda\xi)
  = \int_{\mathbb{R}^d}
    \frac{1 - \cos(\xi \cdot u)}{|u|^{d+1}}\,\lambda\, du
  = \lambda\, J(\xi),
\]
hence $J(\xi) = J(e_1)\,|\xi|$.

It remains to compute $J(e_1)$. Writing $z = (t,w) \in
\mathbb{R} \times \mathbb{R}^{d-1}$,
\[
  J(e_1) = \int_{\mathbb{R}} \int_{\mathbb{R}^{d-1}}
    \frac{1 - \cos t }{(t^2 + |w|^2)^{(d+1)/2}}\, dw\, dt.
\]
Calculating
\[
  \int_{\mathbb{R}^{d-1}}
    \frac{dw}{(t^2 + |w|^2)^{(d+1)/2}}
  = \frac{\pi^{(d-1)/2}}{\Gamma(\frac{d+1}{2})}\,\frac{1}{t^2},
\]
it follows that
\[
  J(e_1) = \frac{\pi^{(d-1)/2}}{\Gamma(\frac{d+1}{2})}
    \int_{\mathbb{R}} \frac{1 - \cos t }{t^2}\, dt
    = \frac{\pi^{(d+1)/2}}{\Gamma(\frac{d+1}{2})} = c_d^{-1}.
  \qedhere
\]
\end{proof}

The fractional gradient term also
admits a Fourier representation.

\begin{lemma}\label{lem:d12-fourier}
For measurable $u : \mathbb{T}^d \to \mathbb{S}^m$,
$\zeta \in C^\infty(\mathbb{T}^d)$, and
$i,j \in \{1,\dots,m{+}1\}$,
\begin{gather*}
  \int_{\mathbb{T}^d} \Big(
    \langle d_{\frac{1}{2}}(P_h^* * u_i), d_{\frac{1}{2}}\zeta
      \rangle_{\mathrm{od}}\, u_j
    - \langle d_{\frac{1}{2}}(P_h^* * u_j), d_{\frac{1}{2}}\zeta
      \rangle_{\mathrm{od}}\, u_i
  \Big)\, dx \\
  = -\frac{1}{2}
    \sum_{k,n \in \mathbb{Z}^d}
    \widehat{P_h^*}(k)\,
    A_{ij}(k,n)\,\widehat{\zeta}(n)\,
    \bigl(|k+n| - |k| - |n|\bigr).
\end{gather*}
\end{lemma}

\begin{proof}
Inserting the Fourier series and interchanging limits by $L^2$
and $L^\infty$ convergence yields
\begin{align*}
  &\int_{\mathbb{T}^d}
    \langle d_{\frac{1}{2}}(P_h^* * u_i), d_{\frac{1}{2}}\zeta
    \rangle_{\mathrm{od}}\, u_j\, dx \\
  &= \frac{c_d}{2}
    \int_{\mathbb{R}^d} \frac{dz}{|z|^{d+1}}
    \sum_{k,n \in \mathbb{Z}^d}
    \widehat{P_h^*}(k)\,
    \widehat{u}_i(k)\,\widehat{\zeta}(n)\,
    \widehat{u}_j(-k-n)\,
    \bigl(e^{i(k+n) \cdot z} - e^{ik \cdot z}
      - e^{in \cdot z} + 1\bigr).
\end{align*}
Subtracting the $(i \leftrightarrow j)$ term, applying the
Fubini--Tonelli theorem, and using Lemma~\ref{lem:integral_identity} gives
the result.

\end{proof}

We can now prove that the approximative antisymmetric term converges
to its continuous counterpart.

\begin{theorem}\label{thm:limit_antisym}
Let $\zeta \in C^\infty(\mathbb{T}^d)$ and
$i,j \in \{1,\dots,m{+}1\}$. Then, as $h \to 0$,
\begin{gather*}
  \sup_{t \in (0,T)} \left|
  \frac{1}{2h} \int_{\mathbb{T}^d} P_h^*(z)
    \int_{\mathbb{T}^d} \bigl[
      ((u_i^h)^z - u_i^h)\,u_j^h - ((u_j^h)^z - u_j^h)\,u_i^h
    \bigr]\,(\zeta^z - \zeta)\, dx\, dz \right. \\
  \left. - \int_{\mathbb{T}^d} \bigl(
    \langle d_{\frac{1}{2}}(P_h^* * u_i^h), d_{\frac{1}{2}}\zeta
      \rangle_{\mathrm{od}}\, u_j^h
    - \langle d_{\frac{1}{2}}(P_h^* * u_j^h), d_{\frac{1}{2}}\zeta
      \rangle_{\mathrm{od}}\, u_i^h
  \bigr)\, dx \right|
  = O(\sqrt{h}).
\end{gather*}
\end{theorem}

\begin{proof}
All bounds below use only $\|u^h\|_{L^2}^2 = (2\pi)^d$ and
\[
  E_h^*(u^h) \;\leq\; E_h^*(u^0) \;\leq\; 2\,E(u^0),
\]
the first by \eqref{eq:energy_dissipation}, the second by
Lemma~\ref{lem:Eh_limit} in the untruncated and
Lemma~\ref{lem:EhN_limit} in the truncated case. Neither quantity
depends on $t$, $h$ or $N$, so the estimates hold uniformly in $t$.
By Lemmas~\ref{lem:fourier-rep} and~\ref{lem:d12-fourier}, applied at
fixed $t$ with $A_{ij}$ formed from the Fourier coefficients of $u^h$,
the difference equals
\[
  \frac{1}{2(2\pi)^{d/2}}
  \sum_{k,n \in \mathbb{Z}^d}
  A_{ij}(k,n)\,\widehat{\zeta}(n)\,
  \Bigl(
    K_h^*(k,n) + (2\pi)^{d/2}\,\widehat{P_h^*}(k)\,
    \bigl(|k+n| - |k| - |n|\bigr)
  \Bigr).
\]

\emph{Untruncated case.}
For $P_h^* = P_h$, the expression in parentheses becomes
\[
  \frac{e^{-h|k+n|} - e^{-h|k|} - e^{-h|n|} + 1}{h}
  + e^{-h|k|}\bigl(|k+n| - |k| - |n|\bigr).
\]
We decompose this into three parts:
\begin{align*}
  I_1(k,n) &= (|k+n| - |k|)(e^{-h|k|} - 1), \\
  I_2(k,n) &= \frac{1 - e^{-h|n|}}{h} - |n|\,e^{-h|k|}, \\
  I_3(k,n) &= (|k+n| - |k|)
    + \frac{e^{-h|k+n|} - 1}{h} + \frac{1 - e^{-h|k|}}{h}.
\end{align*}

\emph{Estimate of $I_1$.}
Using $\bigl||k+n| - |k|\bigr| \leq |n|$ and
$|A_{ij}(k,n)| \leq 2|\widehat{u^h}(k)|\,|\widehat{u^h}(k+n)|$,
\begin{align*}
  &\Bigl|\sum_{k,n} A_{ij}(k,n)\,
    (|k+n| - |k|)(e^{-h|k|} - 1)\,\widehat{\zeta}(n)\Bigr| \\
  &\leq 2 \sum_{k,n}
    |\widehat{u^h}(k)|\,(1 - e^{-h|k|})\,
    |\widehat{u^h}(k+n)|\,|n|\,|\widehat{\zeta}(n)|.
\end{align*}

By Young's convolution inequality,
\[
  \leq 2\Bigl(\sum_n |n|\,|\widehat{\zeta}(n)|\Bigr)
  \Bigl(\sum_k |\widehat{u^h}(k)|^2(1 - e^{-h|k|})^2\Bigr)^{1/2}
  \|u^h\|_{L^2}.
\]
Since $(1 - e^{-h|k|})^2 \leq 2h\,\frac{1 - e^{-h|k|}}{h}$,
we have
\[
  \sum_k |\widehat{u^h}(k)|^2(1 - e^{-h|k|})^2
  \leq 2h \sum_k |\widehat{u^h}(k)|^2\,
    \frac{1 - e^{-h|k|}}{h}
  = 4h\, E_h(u^h),
\]
which is bounded by $4h\, E_h(u^0)$ along the scheme by the
energy-dissipation estimate \eqref{eq:energy_dissipation}.
Hence this term is $O(\sqrt{h})$.

\emph{Estimate of $I_2$.}
Write
\[
  I_2(k,n)
  = \Bigl(\frac{1 - e^{-h|n|}}{h} - |n|\Bigr)
    + |n|\,\bigl(1 - e^{-h|k|}\bigr).
\]
By Taylor expansion, $|1 - e^{-t} - t| \leq \frac{1}{2}t^2$, so the first
bracket is bounded by $\frac{h}{2}|n|^2$ and contributes, by Young's
convolution inequality,
\[
  \Bigl|\sum_{k,n} A_{ij}(k,n)\,
    \Bigl(\frac{1 - e^{-h|n|}}{h} - |n|\Bigr)\,
    \widehat{\zeta}(n)\Bigr|
  \leq h\,\|u^h\|_{L^2}^2\,
    \sum_n |n|^2\,|\widehat{\zeta}(n)|
  = O(h).
\]
The second part is estimated exactly as $I_1$.

\emph{Estimate of $I_3$.}
Define $w_h(x) := x - \frac{1 - e^{-hx}}{h}$, so that
$I_3(k,n) = w_h(|k+n|) - w_h(|k|)$.
Since $w_h'(x) = 1 - e^{-hx}$ and $w_h''(x) = he^{-hx}$,
Taylor expansion and the reverse triangle inequality give
\[
  |w_h(|k+n|) - w_h(|k|)|
  \leq (1 - e^{-h|k|})\,\bigl||k+n| - |k|\bigr|
    + h\,\bigl||k+n| - |k|\bigr|^2.
\]
The first part is estimated as $I_1$ and the second as $I_2$
(using $\bigl||k+n| - |k|\bigr| \leq |n|$).
Hence the total contribution is $O(\sqrt{h})$.

\emph{Truncated case.}
For $P_h^* = P_h^N$, we decompose $K_h^N$ by separating the
indicator functions from the untruncated kernel $K_h$. Since
$E_h(u^h) \leq E_h^N(u^h)$ (as $P_h^N$ has fewer Fourier modes),
the untruncated estimates for $I_1$, $I_2$, $I_3$ remain valid:
the bound $E_h(u^h) \leq E_h^N(u^h) \leq E_h^N(u^0)$ ensures
that the approximative energy controlling the untruncated proof is
still bounded along the truncated scheme.

It therefore suffices to estimate the additional contribution
from the indicator functions, which reduces to
\begin{gather*}
  \frac{1}{2(2\pi)^{d/2}h}
  \sum_{k,n} A_{ij}(k,n)\,\widehat{\zeta}(n)\,\bigl(
    \mathbf{1}_{\{|k+n| > N\}} e^{-h|k+n|}
    - \mathbf{1}_{\{|k| > N\}} e^{-h|k|}
    - \mathbf{1}_{\{|n| > N\}} e^{-h|n|}
  \bigr) \\
  + \frac{1}{2(2\pi)^{d/2}}
  \sum_{k,n} A_{ij}(k,n)\,\widehat{\zeta}(n)\,
    \mathbf{1}_{\{|k| > N\}} e^{-h|k|}\,
    \bigl(|k+n| - |k| - |n|\bigr).
\end{gather*}
For the first sum, each exponential with indicator is bounded by
$e^{-hN}$. Using $|A_{ij}(k,n)| \leq 2|\widehat{u^h}(k)||\widehat{u^h}(k+n)|$
and Young's convolution inequality,
\[
  \sum_{k,n} |A_{ij}(k,n)|\,|\widehat{\zeta}(n)|
  \;\leq\; 2\,\|u^h\|_{L^2}^2\,\sum_n |\widehat{\zeta}(n)|,
\]
which is finite since $\zeta \in C^\infty(\T^d)$. Hence this term is
$O(h^{-1}e^{-hN})$.

For the second sum, bounding
$e^{-h|k|}\mathbf{1}_{\{|k|>N\}} \leq e^{-hN}$, using
$\bigl||k+n| - |k| - |n|\bigr| \leq 2|n|$, and again Young's convolution inequality,
\[
  \sum_{k,n} |A_{ij}(k,n)|\,|\widehat{\zeta}(n)|\,|n|\,
    e^{-h|k|}\,\mathbf{1}_{\{|k|>N\}}
  \;\leq\; 2\,e^{-hN}\,\|u^h\|_{L^2}^2\,\sum_n |n|\,|\widehat{\zeta}(n)|,
\]
so this term is $O(e^{-hN})$.

The condition $hN - d(d+1)\log(1/h) \to \infty$ imposed in
Theorem~\ref{thm:main} (inherited from
Lemma~\ref{lem:PhN_positivity} via positivity of $P_h^N$) gives
$e^{-hN} = o\bigl(h^{d(d+1)}\bigr)$, hence $h^{-1}e^{-hN} = o(h)$ for
every $d \geq 1$. Both extra terms are therefore $o(h)$ and the total is
again $O(\sqrt{h})$.
\end{proof}

\subsection{Proof of Theorem~\ref{thm:main}}

\begin{proof}[Proof of Theorem~\ref{thm:main}]
By the energy-dissipation estimate \eqref{eq:energy_dissipation}
and the estimate \eqref{eq:HEstimate}, the sequence
$P_{h/2}^* * u^h$ is bounded in
$L^\infty(0,T;\dot{H}^{1/2}(\mathbb{T}^d))$ with
$\partial_t^h(P_{h/2}^* * u^h)$ bounded in
$L^2(\mathbb{T}^d \times (0,T))$.
A direct calculation shows equicontinuity in $L^2$:
\[
  \sup_h \|P_{h/2}^* * u^h(\cdot + s, \cdot + y)
    - P_{h/2}^* * u^h\|_{L^2(\mathbb{T}^d\times(0,T))} \to 0
  \quad \text{as } s,y \to 0.
\]

This yields a strongly convergent subsequence
$P_{h_n/2}^* * u^{h_n} \to u$ in $L^2$. Moreover by Jensen's inequality and the positivity of $P^*_h$ we obtain
\begin{align*}
  \|u^{h_n}- P_{h_n}^* * u^{h_n}\|_{L^2}^2
  &= \int_{\mathbb{T}^d}\left(\int_{\mathbb{T}^d} P^*_{h_n}(z)(u^{h_n}(x) - u^{h_n}(x-z)) \,dz\right)^2\, dx\\
  &= \int_{\mathbb{T}^d}P^*_{h_n}(z)\int_{\mathbb{T}^d} \left|u^{h_n}(x) - u^{h_n}(x-z)\right|^2 \,dx\, dz
   \;=\; 4h\,E_{h_n}^*(u^{h_n}) \;\leq\; 4{h_n}\,E_{h_n}^*(u^0) \;\to\; 0 ,
\end{align*}
Hence $u^{h_n} \to u$ in $L^2$ as well.

By Theorem~\ref{thm:limit_antisym}, the approximative conservation
law (Lemma~\ref{lem:approx-EL}) passes to the limit,
and noting that the missing time integral $\int_0^h$ vanishes
(since $u^h = u^0 \in H^{1/2}$ on $[0,h)$), we obtain via
Theorem~\ref{thm:conservation_law_flow} that $u$ is a weak solution
of the half-harmonic map heat flow.

For the truncated scheme with $P_h^N$ in place of $P_h$, the
same argument applies under the condition
$hN - d(d+1)\log(1/h) \to \infty$.

Since $u \in H^1((0,T);L^2(\mathbb{T}^d)) \subset
C([0,T];L^2(\mathbb{T}^d))$, we identify $u(0) = u^0$ by
testing the time derivative with
$\varphi(t,x) = \zeta(t)\,\psi(x)$, where
$\zeta \in C_c^\infty([-\varepsilon,T))$ with $\zeta \equiv 1$
on $[-\varepsilon,0]$ and $\psi \in C^\infty(\mathbb{T}^d)$.
Integration by parts in time gives
\[
  \int_0^T \!\int_{\mathbb{T}^d} \partial_t u\,\varphi\, dx\, dt
  = -\int_0^T \!\int_{\mathbb{T}^d}
    u\,\psi\,\partial_t\zeta\, dx\, dt
  - \int_{\mathbb{T}^d} u(0)\,\psi\, dx.
\]
Using $\partial_t^h(P_{h/2}^* * u^h)
\rightharpoonup \partial_t u$ in $L^2$, the left-hand side equals
\begin{gather*}
  \lim_{h \to 0} \int_0^T \!\int_{\mathbb{T}^d}
    \partial_t^h(P_{h/2}^* * u^h)\,\varphi\, dx\, dt \\
  = -\lim_{h \to 0} \int_0^T \!\int_{\mathbb{T}^d}
    P_{h/2}^* * u^h\,\psi\,\partial_t^{-h}\zeta\, dx\, dt
  - \lim_{h \to 0} \frac{1}{h}\int_0^h \!\int_{\mathbb{T}^d}
    P_{h/2}^* * u^h\,\psi\, dx.
\end{gather*}

By the strong convergence of $P_{h/2}^* * u^h$ in $L^2$
(for the truncated kernel, $P_h^N * u^0 \to u^0$ in $L^2$
by Lemma~\ref{lem:PhN-L2-approx}) and the fact
that $P_{h/2}^* * u^h(t) = P_{h/2}^* * u^0$ for $t \in [0,h)$, this
becomes
\[
  -\int_0^T \!\int_{\mathbb{T}^d}
    u\,\psi\,\partial_t\zeta\, dx\, dt
  - \int_{\mathbb{T}^d} u^0\,\psi\, dx.
\]
Comparing both expressions gives $u(0) = u^0$.

We upgrade the convergence to $H^{1/2}(\mathbb{T}^d)$.
Since $u^{h_n} \to u$ in $L^2(\mathbb{T}^d \times (0,T))$,
passing to a further subsequence (still denoted $h_n$),
$u^{h_n}(t) \to u(t)$ in $L^2(\mathbb{T}^d)$ for a.e.\ $t$.
For the untruncated scheme,
\begin{align*}
  E(u(t))
  &= \lim_{\ell \to 0} E_\ell(u(t))
   = \lim_{\ell \to 0} \lim_{n \to \infty} E_\ell(u^{h_n}(t)) \\
  &\leq \liminf_{n \to \infty} E_{h_n}(u^{h_n}(t))
   \leq \liminf_{n \to \infty} E_{h_n}(u^0)
   \leq E(u^0).
\end{align*}

For the truncated scheme, setting
$R(h,N,u) := \frac{1}{2h}\sum_{|k|>N} e^{-h|k|}|\widehat{u}(k)|^2$,
the Fourier identity \eqref{eq:EmN-identity} gives
\begin{align*}
  E(u(t))
  &= \lim_{\ell \to 0} \lim_{n \to \infty} E_\ell(u^{h_n}(t))
   \;\leq\; \liminf_{n \to \infty} E_{h_n}(u^{h_n}(t)) \\
  &= \liminf_{n \to \infty}
    \bigl[E_{h_n}^{N(h_n)}(u^{h_n}(t))
      - R(h_n, N(h_n), u^{h_n}(t))\bigr] \\
  &\leq \liminf_{n \to \infty}
    \bigl[E_{h_n}^{N(h_n)}(u^0)
      - R(h_n, N(h_n), u^{h_n}(t))\bigr] \\
  &= \liminf_{n \to \infty}
    \bigl[E_{h_n}(u^0) + R(h_n, N(h_n), u^0)
      - R(h_n, N(h_n), u^{h_n}(t))\bigr]
   \;\leq\; E(u^0),
\end{align*}
where both remainders satisfy
\[
  R(h_n, N(h_n), u)
  \leq \frac{e^{-h_n N(h_n)}}{2h_n}\,\|u\|_{L^2}^2
  = \frac{e^{-h_n N(h_n)}}{2h_n}\,(2\pi)^d
  \to 0,
\]
so the last step is Lemma~\ref{lem:Eh_limit}.

This holds for a.e.\ $t$ only, since $u^{h_n}(t) \to u(t)$ in
$L^2(\mathbb{T}^d)$ was obtained for a.e.\ $t$. It extends to every $t$:
since $H^{1/2}$ is reflexive,
$u \in L^\infty(0,T;H^{1/2}) \cap C([0,T];L^2)$ implies weak continuity of $u$
with values in $H^{1/2}$.
Given an arbitrary $t$, choose $t_j \to t$ among the times at which
$E(u(t_j)) \leq E(u^0)$ holds; those times are dense in
$(0,T)$, their complement being a null set. Then
$u(t_j) \rightharpoonup u(t)$ in $H^{1/2}$, and weak lower semicontinuity
gives
\[
  E(u(t)) \;\leq\; \liminf_j E(u(t_j)) \;\leq\; E(u^0).
\]

Since $u \in L^\infty(0,T;\,H^{1/2}(\mathbb{T}^d))$ and
$u(t) \to u^0$ in $L^2(\mathbb{T}^d)$, any sequence
$t_k \to 0$ admits a subsequence with
$u(t_{k_j}) \rightharpoonup u^0$ weakly in $H^{1/2}$.
By weak lower semicontinuity,
$E(u^0) \leq \liminf_{t \to 0} E(u(t))$.
Together, $E(u(t)) \to E(u^0)$, and
with the $L^2$-convergence this gives
$u(t) \to u^0$ in $H^{1/2}(\mathbb{T}^d)$.

\emph{Energy dissipation inequality.}
Since $u^h$, and with it $|P_h^* * u^h|$, is constant in time on each
$I_n = [nh,(n+1)h)$, we may choose one set per time step,
\[
  \Gamma_n := \bigl\{ x \in \mathbb{T}^d :\
    |P_h^* * u^n|(x) > \tfrac12 \bigr\} .
\]
Given $m = m(h) \in \mathbb{N}$, we group $m$ consecutive steps into
blocks: with $\ell := mh$, $J := \lceil T/\ell \rceil$ and
$J_j := [j\ell,(j+1)\ell) \cap [0,T)$ of length $|J_j| \leq \ell$ for
$j = 0,\dots,J-1$, we set
\begin{equation}\label{eq:blockcutoff}
  O_j := \bigcap_{n \,:\, nh \in J_j} \Gamma_n ,
  \qquad
  A_h := \bigcup_{j=0}^{J-1} O_j \times J_j.
\end{equation}

Let $\varphi \in C_c^\infty(\mathbb{T}^d \times (0,T))$, with $h$ small
enough that $\operatorname{supp}\varphi + [-h,h] \subset (0,T)$.
Splitting $u^h = P_{h/2}^* * u^h + (u^h - P_{h/2}^* * u^h)$, applying
discrete integration by parts in $t$ to the second part, and using
$\partial_t^{-h}(\mathbf{1}_{A_h}\varphi)
 = \mathbf{1}_{A_h}\,\partial_t^{-h}\varphi
 + \varphi(\cdot\,,t-h)\,\partial_t^{-h}\mathbf{1}_{A_h}$, we obtain
\[
  \int_0^T\!\!\int_{\mathbb{T}^d} \mathbf{1}_{A_h}\,\varphi\,\partial_t^h u^h\,dx\,dt
  \;=\; \int_0^T\!\!\int_{\mathbb{T}^d}
    \varphi\,\partial_t^h(P_{h/2}^* * u^h)\,dx\,dt
   \;-\; I_1(h) \;-\; I_2(h) \;-\; I_3(h),
\]
where
\begin{align*}
  I_1(h) &:= \int_0^T\!\!\int_{\mathbb{T}^d}
    (1-\mathbf{1}_{A_h})\,\varphi\,\partial_t^h(P_{h/2}^* * u^h)\,dx\,dt , \\
  I_2(h) &:= \int_0^T\!\!\int_{\mathbb{T}^d}
    \bigl(u^h - P_{h/2}^* * u^h\bigr)\,\mathbf{1}_{A_h}\,
    \partial_t^{-h}\varphi\,dx\,dt , \\
  I_3(h) &:= \int_0^T\!\!\int_{\mathbb{T}^d}
    \bigl(u^h - P_{h/2}^* * u^h\bigr)\,\varphi(\cdot\,,t-h)\,
    \partial_t^{-h}\mathbf{1}_{A_h}\,dx\,dt .
\end{align*}

By the Cauchy--Schwarz inequality and
$\|\partial_t^h(P_{h/2}^* * u^h)\|_{L^2} \leq (2E_h^*(u^0))^{1/2}$ from
Theorem~\ref{thm:energy-identity},
\[
  |I_1(h)|
  \;\leq\; \bigl\|(1-\mathbf{1}_{A_h})\,\varphi\bigr\|_{L^2}\,
    \bigl\|\partial_t^h(P_{h/2}^* * u^h)\bigr\|_{L^2}
  \;\leq\; \|\varphi\|_{L^\infty}\,|A_h^c|^{1/2}\,
    \bigl(2E_h^*(u^0)\bigr)^{1/2}.
\]

Likewise, by the Cauchy--Schwarz inequality, $|\mathbf{1}_{A_h}| \leq 1$ and
$\sup_t \|u^h(t) - P_{h/2}^* * u^h(t)\|_{L^2}
 \leq (2h\,E_h^*(u^0))^{1/2}$,
\[
  |I_2(h)|
  \;\leq\; \bigl\|u^h - P_{h/2}^* * u^h\bigr\|_{L^2}\,
    \bigl\|\mathbf{1}_{A_h}\,\partial_t^{-h}\varphi\bigr\|_{L^2}
  \;\leq\; \bigl(2T\,E_h^*(u^0)\bigr)^{1/2}
    \bigl(T\,|\mathbb{T}^d|\bigr)^{1/2}\,
    \|\partial_t\varphi\|_{L^\infty}\;h^{1/2}.
\]

For $I_3(h)$ the shift falls on $\mathbf{1}_{A_h}$. If $t$ and $t-h$ lie in the
same block then $\mathbf{1}_{A_h}(\cdot\,,t) = \mathbf{1}_{A_h}(\cdot\,,t-h)$, so for
$h \leq \ell$ the integrand vanishes outside the $J-1$ slabs
$\mathbb{T}^d \times [j\ell,\,j\ell+h)$, on which
$|\partial_t^{-h}\mathbf{1}_{A_h}| = \tfrac1h\,\mathbf{1}_{O_j \triangle O_{j-1}}$.
Using $|\varphi(\cdot\,,t-h)| \leq \|\varphi\|_{L^\infty}$,
\[
  |I_3(h)|
  \;\leq\; \frac{\|\varphi\|_{L^\infty}}{h}
    \sum_{j=1}^{J-1} \int_{j\ell}^{j\ell+h}\!\!
      \int_{O_j \triangle O_{j-1}}
      \bigl|u^h - P_{h/2}^* * u^h\bigr|\,dx\,dt .
\]
For fixed $t$, the Cauchy--Schwarz inequality against the characteristic function of
$O_j \triangle O_{j-1}$ gives
\begin{align*}
  \int_{O_j \triangle O_{j-1}}
    \bigl|u^h(t) - P_{h/2}^* * u^h(t)\bigr|\,dx
  \;\leq\; |O_j \triangle O_{j-1}|^{1/2}
    \bigl(2h\,E_h^*(u^0)\bigr)^{1/2},
\end{align*}
Hence
\begin{align*}
  |I_3(h)|
  \;\leq\; \|\varphi\|_{L^\infty}\bigl(2E_h^*(u^0)\bigr)^{1/2}\,h^{1/2}
    \sum_{j=1}^{J-1} |O_j \triangle O_{j-1}|^{1/2} .
\end{align*}

It remains to estimate the size of the sets $|O_j \triangle O_{j-1}|$. On $\Gamma_n^c$ we have
$1 - u^n\cdot(P_h^* * u^n) \geq 1 - |P_h^* * u^n| \geq \tfrac12$, and
the integrand is nonnegative, so
\[
  \tfrac12\,|\Gamma_n^c|
  \;\leq\; \int_{\mathbb{T}^d}
    \bigl(1 - u^n\cdot(P_h^* * u^n)\bigr)\,dx
  \;=\; 2h\,E_h^*(u^n)
  \;\leq\; 2h\,E_h^*(u^0) .
\]
Hence, with $\ell = mh$,
\[
  |O_j^c| \;=\; \Bigl|\bigcup_{n\,:\,nh \in J_j}\Gamma_n^c\Bigr|
  \;\leq\; \sum_{n\,:\,nh \in J_j}|\Gamma_n^c|
  \;\leq\; 4mh\,E_h^*(u^0)
  \;=\; 4\ell\,E_h^*(u^0),
\]
\[
  |O_j \triangle O_{j-1}| \;\leq\; |O_j^c| + |O_{j-1}^c|
  \;\leq\; 8\ell\,E_h^*(u^0),
\]
\[
  |A_h^c| \;=\; \sum_{j=0}^{J-1} |J_j|\,|O_j^c|
  \;\leq\; 4\ell\,E_h^*(u^0) \sum_{j=0}^{J-1} |J_j|
  \;=\; 4\,T\ell\,E_h^*(u^0) .
\]

Inserting the measure bounds,
\[
  |I_1(h)| \;\leq\; \|\varphi\|_{L^\infty}
    \bigl(4T\ell\,E_h^*(u^0)\bigr)^{1/2}
    \bigl(2E_h^*(u^0)\bigr)^{1/2}
  \;=\; 2\sqrt2\,E_h^*(u^0)\,\|\varphi\|_{L^\infty}\,(T\,mh)^{1/2},
\]
\begin{align*}
  |I_3(h)|
  &\;\leq\; \|\varphi\|_{L^\infty}\bigl(2E_h^*(u^0)\bigr)^{1/2}h^{1/2}
    \sum_{j=1}^{J-1}|O_j \triangle O_{j-1}|^{1/2} \\
  &\;\leq\; \|\varphi\|_{L^\infty}\bigl(2E_h^*(u^0)\bigr)^{1/2}h^{1/2}\,
    \frac{T}{\ell}\,\bigl(8\ell\,E_h^*(u^0)\bigr)^{1/2}
  \;=\; 4\,E_h^*(u^0)\,T\,\|\varphi\|_{L^\infty}\,m^{-1/2} .
\end{align*}
We choose $m = \lceil h^{-1/2}\rceil$, so that $m \to \infty$ and
$mh \to 0$; then $I_1(h), I_2(h), I_3(h) \to 0$ as $h \to 0$.

It remains to replace $\mathbf{1}_{A_h}\,\partial_t^h u^h$ by
$\mathbf{1}_{A_h}\,|P_h^* * u^h|^{1/2}\,\partial_t^h u^h$. Inserting
$1 = |P_h^* * u^h|^{-1/2}\,|P_h^* * u^h|^{1/2}$ and applying
the Cauchy--Schwarz inequality,
\begin{align*}
  \Bigl|\int_0^T\!\!\int_{\mathbb{T}^d}
    \mathbf{1}_{A_h}\bigl(1-|P_h^* * u^h|^{1/2}\bigr)\,
    \partial_t^h u^h\,\varphi\,dx\,dt\Bigr|
  \;\leq\;& \|\varphi\|_{L^\infty}
    \Bigl(\int_0^T\!\!\int_{\mathbb{T}^d} \mathbf{1}_{A_h}\,
      \frac{\bigl(1-|P_h^* * u^h|^{1/2}\bigr)^2}{|P_h^* * u^h|}
      \,dx\,dt\Bigr)^{1/2} \\
  &\times \Bigl(\int_0^T\!\!\int_{\mathbb{T}^d}
    |P_h^* * u^h|\,|\partial_t^h u^h|^2\,dx\,dt\Bigr)^{1/2} .
\end{align*}
The second factor is at most $(2E_h^*(u^0))^{1/2}$ by
Theorem~\ref{thm:energy-identity}. We have, for $a > \tfrac12$,
\[
  \frac{(1-a^{1/2})^2}{a}
  \;=\; \frac1a\Bigl(\frac{1-a}{1+a^{1/2}}\Bigr)^{2}
  \;\leq\; \frac{(1-a)^2}{a}
  \;\leq\; 2\,(1-a) ,
\]
so that, using $u^h\cdot(P_h^* * u^h) \leq |P_h^* * u^h|$ and $|P_h^* * u^h| > \frac{1}{2}$ in $A_h$, we obtain
\[
  \int_0^T\!\!\int_{\mathbb{T}^d}
    \bigl(1-|P_h^* * u^h|\bigr)\,dx\,dt
  \;\leq\; \int_0^T\!\!\int_{\mathbb{T}^d}
    \bigl(1 - u^h\cdot(P_h^* * u^h)\bigr)\,dx\,dt
  \;\leq\; 2\,T h\,E_h^*(u^0) .
\]
Altogether the difference is bounded by
$2\sqrt2\,E_h^*(u^0)\,\|\varphi\|_{L^\infty}\,(Th)^{1/2} \to 0$.

Since $C_c^\infty(\mathbb{T}^d \times (0,T))$ is dense in
$L^2(\mathbb{T}^d \times (0,T))$ and
$\|\mathbf{1}_{A_h}\,|P_h^* * u^h|^{1/2}\,\partial_t^h u^h\|_{L^2}^2
 \leq 2E_h^*(u^0)$ uniformly in $h$, this gives
\[
  \mathbf{1}_{A_h}\,|P_h^* * u^h|^{1/2}\,\partial_t^h u^h
  \;\rightharpoonup\; \partial_t u
  \qquad\text{in } L^2(\mathbb{T}^d \times (0,T)) .
\]
Fix $t$ and set $k = \lfloor t/h \rfloor$, so that $t_k = kh \leq t$ and
$|t - t_k| \leq h \to 0$. Since $\mathbf{1}_{A_h} \leq 1$ and the integrand is
nonnegative,
\[
  \frac12\int_0^{t_k}\!\!\int_{\mathbb{T}^d}
    |P_h^* * u^h|\,|\partial_t^h u^h|^2\,dx\,ds
  \;\geq\; \frac12\int_0^{t_k}\!\!\int_{\mathbb{T}^d}
    \mathbf{1}_{A_h}\,|P_h^* * u^h|\,|\partial_t^h u^h|^2\,dx\,ds ,
\]
and since $t_k \to t$, weak lower semicontinuity gives
\[
  \liminf_{h\to0}\ \frac12\int_0^{t_k}\!\!\int_{\mathbb{T}^d}
    |P_h^* * u^h|\,|\partial_t^h u^h|^2\,dx\,ds
  \;\geq\; \frac12\int_0^{t}\!\!\int_{\mathbb{T}^d}
    |\partial_t u|^2\,dx\,ds .
\]
Likewise
$\liminf_{h\to0}\tfrac12\int_0^{t_k}\|\partial_t^h(P_{h/2}^**u^h)\|_{L^2}^2\,ds
\ge \tfrac12\int_0^t\|\partial_t u\|_{L^2}^2\,ds$. With $E(u(t))\le\liminf E_h^*(u^h(t_k))$
and $E_h^*(u^0)\to E(u^0)$, the $\liminf$ of
Theorem~\ref{thm:energy-identity} yields
\[
  E(u(t))+\int_0^t\|\partial_t u\|_{L^2}^2\,ds\le E(u^0).\qedhere
\]
\end{proof}

%% file: weak_strong.tex

\begin{proof}[Proof of Theorem~\ref{thm:wsu-main}]
  Let $\tau(u) := -(-\Delta)^{1/2}u + |d_{\frac{1}{2}}u|^2_{\mathrm{od}}\,u$
  denote the tension field, so that $\partial_t u = \tau(u)$.

  We estimate
  $E(u-v)(t) = E(u(t)) + E(v(t))
  - \langle d_{\frac{1}{2}}u(t),\,d_{\frac{1}{2}}v(t)\rangle_{L^2_{\mathrm{od}}}$.
  By the energy inequalities for $u$ and $v$,
  \begin{equation}\label{eq:hhmhf-wsu-start}
    E(u-v)(t)
    \;\leq\;
    2E(u^0)
    - \int_0^t \bigl(\|\tau(u)\|_{L^2}^2 + \|\tau(v)\|_{L^2}^2\bigr)\,ds
    - \langle d_{\frac{1}{2}}u(t),\,d_{\frac{1}{2}}v(t)\rangle_{L^2_{\mathrm{od}}}.
  \end{equation}

  We differentiate the cross term. By the integration by parts formula,
  \[
    \langle d_{\frac{1}{2}}u,\,d_{\frac{1}{2}}v\rangle_{L^2_{\mathrm{od}}}
    = \langle u,\,(-\Delta)^{1/2}v\rangle_{L^2} .
  \]
  By the product rule,
  \[
    \partial_t\langle d_{\frac{1}{2}}u,\,d_{\frac{1}{2}}v\rangle_{L^2_{\mathrm{od}}}
    = \langle u_t,\,(-\Delta)^{1/2}v\rangle_{L^2}
     + \langle d_{\frac{1}{2}}u,\,d_{\frac{1}{2}}v_t\rangle_{L^2_{\mathrm{od}}}.
  \]
  Integrating in time and using $u|_{t=0} = v|_{t=0} = u^0$:
  \[
    \langle d_{\frac{1}{2}}u(t),\,d_{\frac{1}{2}}v(t)\rangle_{L^2_{\mathrm{od}}}
    = 2E(u^0)
    + \int_0^t\bigl(
      \langle u_s,\,(-\Delta)^{1/2}v\rangle_{L^2}
      + \langle d_{\frac{1}{2}}u,\,d_{\frac{1}{2}}v_s\rangle_{L^2_{\mathrm{od}}}
    \bigr)\,ds.
  \]
  Substituting into \eqref{eq:hhmhf-wsu-start}, the $2E(u^0)$ cancels:
  \begin{equation}\label{eq:hhmhf-wsu-cross}
  \begin{split}
    E(u-v)(t)
    \;\leq\;
    &- \int_0^t \bigl(\|\tau(u)\|_{L^2}^2 + \|\tau(v)\|_{L^2}^2\bigr)\,ds \\
    &- \int_0^t\bigl(
      \langle u_s,\,(-\Delta)^{1/2}v\rangle_{L^2}
      + \langle d_{\frac{1}{2}}u,\,d_{\frac{1}{2}}v_s\rangle_{L^2_{\mathrm{od}}}
    \bigr)\,ds.
  \end{split}
  \end{equation}

  We expand the two integrands separately.
  For the first, insert $u_s = \tau(u)$ and use \eqref{eq:HHMHF}, i.e.,
  $(-\Delta)^{1/2}v = |d_{\frac{1}{2}}v|^2_{\mathrm{od}}v - v_s$,
  to obtain
  \[
    \langle u_s,\,(-\Delta)^{1/2}v\rangle_{L^2}
    = \langle|d_{\frac{1}{2}}v|^2_{\mathrm{od}},\,v\cdot u_s\rangle_{L^2}
      - \langle\tau(u),\,\tau(v)\rangle_{L^2}.
  \]
  For the second, take $\varphi = v_s$ in the weak formulation
  \eqref{eq:weak-hhmhf-main} of $u$, which is admissible since
  $v_s \in H^{1/2}\cap L^\infty(\mathbb{T}^d)$ for a.e.\ $s$:
  \[
    \langle d_{\frac{1}{2}}u,\,d_{\frac{1}{2}}v_s\rangle_{L^2_{\mathrm{od}}}
    = \langle|d_{\frac{1}{2}}u|^2_{\mathrm{od}},\,u\cdot v_s\rangle_{L^2}
      - \langle\tau(u),\,\tau(v)\rangle_{L^2}.
  \]
  Adding, substituting into \eqref{eq:hhmhf-wsu-cross}, and using
  \[-\|\tau(u)\|^2 - \|\tau(v)\|^2 + 2\langle\tau(u),\tau(v)\rangle
  = -\|\tau(u)-\tau(v)\|^2\]
  we have that
  \begin{equation}\label{eq:hhmhf-wsu-tau}
    E(u-v)(t)
    \;\leq\;
    - \int_0^t \|\tau(u)-\tau(v)\|_{L^2}^2\,ds
    - \int_0^t\Bigl(
      \langle|d_{\frac{1}{2}}u|^2_{\mathrm{od}},\,u\cdot v_s\rangle_{L^2}
      + \langle|d_{\frac{1}{2}}v|^2_{\mathrm{od}},\,v\cdot u_s\rangle_{L^2}
    \Bigr)\,ds.
  \end{equation}

  We simplify the nonlinear terms in \eqref{eq:hhmhf-wsu-tau}.
  Since $|v|^2=1$ we have $v\cdot v_s=0$, and since $|u|^2=1$ we have
  $u\cdot u_s=0$. Therefore
  \begin{align*}
    &\langle|d_{\frac{1}{2}}u|^2_{\mathrm{od}},\,u\cdot v_s\rangle_{L^2}
     + \langle|d_{\frac{1}{2}}v|^2_{\mathrm{od}},\,v\cdot u_s\rangle_{L^2} \\
    &= \langle|d_{\frac{1}{2}}u|^2_{\mathrm{od}},\,(u-v)\cdot v_s\rangle_{L^2}
     - \langle|d_{\frac{1}{2}}v|^2_{\mathrm{od}},\,(u-v)\cdot u_s\rangle_{L^2} \\
    &= \langle|d_{\frac{1}{2}}u|^2_{\mathrm{od}}-|d_{\frac{1}{2}}v|^2_{\mathrm{od}},\,
         (u-v)\cdot v_s\rangle_{L^2}
     + \langle|d_{\frac{1}{2}}v|^2_{\mathrm{od}},\,(u-v)\cdot(v_s-u_s)\rangle_{L^2} \\
    &= \langle|d_{\frac{1}{2}}u|^2_{\mathrm{od}}-|d_{\frac{1}{2}}v|^2_{\mathrm{od}},\,
         (u-v)\cdot v_s\rangle_{L^2}
     - \langle|d_{\frac{1}{2}}v|^2_{\mathrm{od}},\,
         \partial_s\tfrac{|u-v|^2}{2}\rangle_{L^2},
  \end{align*}
  where in the first step we used $v\cdot v_s = 0$ and $u\cdot u_s = 0$,
  in the second we added and subtracted
  $\langle|d_{\frac{1}{2}}v|^2_{\mathrm{od}},\,(u-v)\cdot v_s\rangle_{L^2}$,
  and in the third we used $(u-v)\cdot(u_s-v_s) = \partial_s\tfrac{|u-v|^2}{2}$.

  Substituting back into \eqref{eq:hhmhf-wsu-tau} and integrating the last
  term by parts in time, using $u|_{s=0}=v|_{s=0}=u^0$:
  \begin{align*}
    \int_0^t\langle|d_{\frac{1}{2}}v|^2_{\mathrm{od}},\,
      \partial_s\tfrac{|u-v|^2}{2}\rangle_{L^2}\,ds
    &= \Bigl[\langle|d_{\frac{1}{2}}v|^2_{\mathrm{od}},\,
        \tfrac{|u-v|^2}{2}\rangle_{L^2}\Bigr]_0^t
     - \int_0^t\langle
         \partial_s|d_{\frac{1}{2}}v|^2_{\mathrm{od}},\,
         \tfrac{|u-v|^2}{2}\rangle_{L^2}\,ds \\
    &= \langle|d_{\frac{1}{2}}v(t)|^2_{\mathrm{od}},\,
        \tfrac{|u(t)-v(t)|^2}{2}\rangle_{L^2} \\
    &\quad - \int_0^t\langle
         \langle d_{\frac{1}{2}}v,\,d_{\frac{1}{2}}v_s\rangle_{\mathrm{od}},\,|u-v|^2
       \rangle_{L^2}\,ds,
  \end{align*}
  Hence, dropping the non-positive $\tau$-term:
  \begin{align}\label{eq:hhmhf-wsu-ibp}
    E(u-v)(t)
    \;\leq\;&
    \langle|d_{\frac{1}{2}}v(t)|^2_{\mathrm{od}},\,\tfrac{|u(t)-v(t)|^2}{2}\rangle_{L^2}
    \notag\\
    &- \int_0^t\langle|d_{\frac{1}{2}}u|^2_{\mathrm{od}}-|d_{\frac{1}{2}}v|^2_{\mathrm{od}},\,
        (u-v)\cdot v_s\rangle_{L^2}\,ds
    \notag\\
    &- \int_0^t\langle
        \langle d_{\frac{1}{2}}v,\,d_{\frac{1}{2}}v_s\rangle_{\mathrm{od}},\,|u-v|^2
      \rangle_{L^2}\,ds.
  \end{align}

  Set $\mathcal{E}_{1/2}(t)
  := E(u-v)(t) + \tfrac{1}{2}\|u(t)-v(t)\|_{L^2}^2$
  and $w := u-v$.  Since $|u|=|v|=1$ we have $|w|\leq 2$ and
  \begin{equation}\label{eq:hhmhf-grad-split}
    |d_{\frac{1}{2}}u|^2_{\mathrm{od}} - |d_{\frac{1}{2}}v|^2_{\mathrm{od}}
    = |d_{\frac{1}{2}}w|^2_{\mathrm{od}} + 2\,\langle d_{\frac{1}{2}}v,\,d_{\frac{1}{2}}w\rangle_{\mathrm{od}}.
  \end{equation}
  Since $v$ is a strong solution, the constants
  \[
    M_1 := \|d_{\frac{1}{2}}v\|_{L^\infty(0,T;\,L^\infty_{\mathrm{od}})},
    \qquad
    M_2 := \|v_s\|_{L^\infty((0,T)\times\mathbb{T}^d)},
    \qquad
    M_3 := \|d_{\frac{1}{2}}v_s\|_{L^\infty(0,T;\,L^\infty_{\mathrm{od}})}
  \]
  are finite.

  \textit{Estimate of the boundary term.}
  By Hölder's inequality and $\tfrac{1}{2}\|w\|_{L^2}^2\leq\mathcal{E}_{1/2}$:
  \begin{equation}\label{eq:hhmhf-wsu-bdy}
    \langle|d_{\frac{1}{2}}v(t)|^2_{\mathrm{od}},\,\tfrac{|w|^2}{2}\rangle_{L^2}
    \leq M_1^2\,\tfrac{1}{2}\|w(t)\|_{L^2}^2
    \leq M_1^2\,\mathcal{E}_{1/2}(t).
  \end{equation}

  \textit{Estimate of the middle integral.}
  Inserting \eqref{eq:hhmhf-grad-split} and using $|w|\leq 2$:
  \begin{align*}
    &\bigl|\langle|d_{\frac{1}{2}}u|^2_{\mathrm{od}}-|d_{\frac{1}{2}}v|^2_{\mathrm{od}},\,
       w\cdot v_s\rangle_{L^2}\bigr| \\
    &\leq M_2\|w\|_{L^\infty}\|d_{\frac{1}{2}}w\|_{L^2_{\mathrm{od}}}^2
     + 2M_1 M_2\|d_{\frac{1}{2}}w\|_{L^2_{\mathrm{od}}}\|w\|_{L^2} \\
    &\leq 2M_2\cdot 2E(u-v)
     + M_1 M_2\bigl(\|d_{\frac{1}{2}}w\|_{L^2_{\mathrm{od}}}^2 + \|w\|_{L^2}^2\bigr),
  \end{align*}
  using Young's inequality.  Hence
  \begin{equation}\label{eq:hhmhf-wsu-mid}
    \int_0^t\bigl|\langle|d_{\frac{1}{2}}u|^2_{\mathrm{od}}-|d_{\frac{1}{2}}v|^2_{\mathrm{od}},\,
      w\cdot v_s\rangle_{L^2}\bigr|\,ds
    \leq (4M_2+2M_1 M_2)\int_0^t\mathcal{E}_{1/2}(s)\,ds.
  \end{equation}

  \textit{Estimate of the last integral.}
  Since $v$ is a strong solution, we have
  $\|d_{\frac{1}{2}}v\|_{L^\infty_{\mathrm{od}}}\leq M_1$ and
  $\|d_{\frac{1}{2}}v_s\|_{L^\infty_{\mathrm{od}}}\leq M_3$, hence
  \begin{align*}
    \bigl|\langle\langle d_{\frac{1}{2}}v,\,d_{\frac{1}{2}}v_s\rangle_{\mathrm{od}},\,
      |w|^2\rangle_{L^2}\bigr|
    &\leq M_1 M_3\,\|w\|_{L^2}^2
    \leq 2M_1 M_3\,\mathcal{E}_{1/2}.
  \end{align*}
  Therefore
  \begin{equation}\label{eq:hhmhf-wsu-last}
    \int_0^t\bigl|\langle\langle d_{\frac{1}{2}}v,\,d_{\frac{1}{2}}v_s\rangle_{\mathrm{od}},\,|w|^2
      \rangle_{L^2}\bigr|\,ds
    \leq 2M_1 M_3\int_0^t\mathcal{E}_{1/2}(s)\,ds.
  \end{equation}

  \textit{Estimate of $\partial_t\frac{1}{2}\|w\|_{L^2}^2$.}
  Since $u$ satisfies \eqref{eq:weak-hhmhf-main} with test function $w$
  and $v$ satisfies the \eqref{eq:weak-hhmhf-main} pointwise,
  \begin{align*}
    \partial_t\tfrac{1}{2}\|w\|_{L^2}^2
    &= \langle u_t-v_t,\,w\rangle_{L^2} \\
    &= -\langle d_{\frac{1}{2}}u,\,d_{\frac{1}{2}}w\rangle_{L^2_{\mathrm{od}}}
     + \langle|d_{\frac{1}{2}}u|^2_{\mathrm{od}}u,\,w\rangle_{L^2}
     + \langle d_{\frac{1}{2}}v,\,d_{\frac{1}{2}}w\rangle_{L^2_{\mathrm{od}}}
     - \langle|d_{\frac{1}{2}}v|^2_{\mathrm{od}}v,\,w\rangle_{L^2} \\
    &= -\|d_{\frac{1}{2}}w\|_{L^2_{\mathrm{od}}}^2
     + \underbrace{
         \langle|d_{\frac{1}{2}}u|^2_{\mathrm{od}}w,\,w\rangle_{L^2}
       }_{=:I_1}
     + \underbrace{
         \langle(|d_{\frac{1}{2}}u|^2_{\mathrm{od}}-|d_{\frac{1}{2}}v|^2_{\mathrm{od}})v,\,
         w\rangle_{L^2}
       }_{=:I_2}.
  \end{align*}
  For $I_1$, expand $|d_{\frac{1}{2}}u|^2_{\mathrm{od}}
  = |d_{\frac{1}{2}}(w+v)|^2_{\mathrm{od}}$ using \eqref{eq:hhmhf-grad-split}:
  \begin{align*}
    I_1 &= \int_{\mathbb{T}^d}|d_{\frac{1}{2}}u|^2_{\mathrm{od}}\,|w|^2\,dx \\
        &= \int_{\mathbb{T}^d}|d_{\frac{1}{2}}w|^2_{\mathrm{od}}\,|w|^2\,dx
         + 2\int_{\mathbb{T}^d}\langle d_{\frac{1}{2}}v,\,d_{\frac{1}{2}}w\rangle_{\mathrm{od}}\,|w|^2\,dx
         + \int_{\mathbb{T}^d}|d_{\frac{1}{2}}v|^2_{\mathrm{od}}\,|w|^2\,dx \\
        &\leq \|w\|_{L^\infty}^2\|d_{\frac{1}{2}}w\|_{L^2_{\mathrm{od}}}^2
         + 2M_1\|d_{\frac{1}{2}}w\|_{L^2_{\mathrm{od}}}\|w\|_{L^\infty}\|w\|_{L^2}
         + M_1^2\|w\|_{L^2}^2 \\
        &\leq 4\|d_{\frac{1}{2}}w\|_{L^2_{\mathrm{od}}}^2
         + 4M_1\|d_{\frac{1}{2}}w\|_{L^2_{\mathrm{od}}}\|w\|_{L^2}
         + M_1^2\|w\|_{L^2}^2,
  \end{align*}
  using $\|w\|_{L^\infty}\leq 2$.  Applying Young's inequality
  $4M_1 ab\leq 2M_1(a^2+b^2)$:
  \[
    I_1 \leq (4+2M_1)\|d_{\frac{1}{2}}w\|_{L^2_{\mathrm{od}}}^2
         + (2M_1+M_1^2)\|w\|_{L^2}^2
    \leq C_1(v)\,\mathcal{E}_{1/2},
  \]
  where $C_1(v) := 2(4+4M_1+M_1^2)$.

  For $I_2$, use \eqref{eq:hhmhf-grad-split}, $|v|=1$, and $|w|\leq 2$:
  \begin{align*}
    I_2 &= \langle|d_{\frac{1}{2}}w|^2_{\mathrm{od}}v
           + 2\langle d_{\frac{1}{2}}v,\,d_{\frac{1}{2}}w\rangle_{\mathrm{od}}v,\,w\rangle_{L^2} \\
        &\leq \|w\|_{L^\infty}\|d_{\frac{1}{2}}w\|_{L^2_{\mathrm{od}}}^2
         + 2M_1\|d_{\frac{1}{2}}w\|_{L^2_{\mathrm{od}}}\|w\|_{L^2} \\
        &\leq 2\|d_{\frac{1}{2}}w\|_{L^2_{\mathrm{od}}}^2
         + M_1\bigl(\|d_{\frac{1}{2}}w\|_{L^2_{\mathrm{od}}}^2+\|w\|_{L^2}^2\bigr)
    \;\leq\; C_2(v)\,\mathcal{E}_{1/2},
  \end{align*}
  where $C_2(v) := 2(2+M_1)$.
  Dropping $-\|d_{\frac{1}{2}}w\|_{L^2_{\mathrm{od}}}^2 \leq 0$ and combining:
  \begin{equation}\label{eq:hhmhf-wsu-L2}
    \partial_t\tfrac{1}{2}\|w(t)\|_{L^2}^2
    \;\leq\; C_3(v)\,\mathcal{E}_{1/2}(t),
    \quad C_3(v) := C_1(v)+C_2(v).
  \end{equation}
  Integrating from $0$ to $t$ with $w|_{t=0}=0$:
  \begin{equation}\label{eq:hhmhf-wsu-L2int}
    \tfrac{1}{2}\|w(t)\|_{L^2}^2
    \;\leq\; C_3(v)\int_0^t\mathcal{E}_{1/2}(s)\,ds.
  \end{equation}

  \textit{Closing the Grönwall inequality.}
  Substituting \eqref{eq:hhmhf-wsu-bdy}, \eqref{eq:hhmhf-wsu-mid},
  \eqref{eq:hhmhf-wsu-last} into \eqref{eq:hhmhf-wsu-ibp} and
  applying \eqref{eq:hhmhf-wsu-L2int}:
  \[
    E(u-v)(t)
    \leq M_1^2\,\tfrac{1}{2}\|w(t)\|_{L^2}^2
    + (4M_2+2M_1 M_2+2M_1 M_3)\int_0^t\mathcal{E}_{1/2}(s)\,ds
    \leq C(v)\int_0^t\mathcal{E}_{1/2}(s)\,ds,
  \]
  where $C(v) := M_1^2 C_3(v)+4M_2+2M_1 M_2+2M_1 M_3$.
  Adding \eqref{eq:hhmhf-wsu-L2int}:
  \begin{equation}\label{eq:hhmhf-wsu-gronwall}
    \mathcal{E}_{1/2}(t)
    \;\leq\; \bigl(C(v)+C_3(v)\bigr)\int_0^t\mathcal{E}_{1/2}(s)\,ds.
  \end{equation}
  Since $\mathcal{E}_{1/2}(0) = 0$ (as $u|_{t=0}=v|_{t=0}=u^0$),
  Grönwall's lemma applied to \eqref{eq:hhmhf-wsu-gronwall} gives
  $\mathcal{E}_{1/2}(t) = 0$ for all $t\in[0,T]$.
  Hence $E(u-v)(t) = 0$ and $\|u(t)-v(t)\|_{L^2} = 0$,
  which gives $u = v$ a.e.\ on $\mathbb{T}^d\times(0,T)$.
\end{proof}

